\documentclass[11pt,a4paper]{article}

\usepackage[utf8]{inputenc}
\usepackage[T1]{fontenc}
\usepackage[english]{babel}
\usepackage{amsmath}
\usepackage{amsfonts}
\usepackage{amsthm}
\usepackage{amssymb}
\usepackage{makeidx}
\usepackage{graphicx}
\usepackage{lmodern}
\usepackage[left=2cm,right=2cm,top=2cm,bottom=2cm]{geometry}

\usepackage{color}
\usepackage{comment}
\usepackage[dvipsnames]{xcolor}
\usepackage{graphicx}
\usepackage{framed}
\usepackage{tikz}
\usepackage{calrsfs}
\DeclareMathAlphabet{\pazocal}{OMS}{zplm}{m}{n}

\usepackage{fourier-orns}
\usepackage{mathtools}
\usepackage{relsize}
\usepackage{dsfont}
\usepackage{comment}
\usepackage{hyperref}

\newcommand{\B}{\mathbb{B}}
\newcommand{\R}{\mathbb{R}}

\renewcommand{\H}{\mathbb{H}}
\newcommand{\J}{\mathbb{J}}

\newcommand{\A}{\pazocal{A}}
\newcommand{\F}{\pazocal{F}}
\newcommand{\G}{\pazocal{G}}
\newcommand{\U}{\pazocal{U}}

\newcommand{\V}{\pazocal{V}}
\newcommand{\M}{\pazocal{M}}
\newcommand{\Bpazo}{\pazocal{B}}
\newcommand{\Qpazo}{\pazocal{Q}}
\newcommand{\Hpazo}{\pazocal{H}}
\newcommand{\Npazo}{\pazocal{N}}
\newcommand{\Lpazo}{\pazocal{L}}
\newcommand{\Ppazo}{\pazocal{P}}

\newcommand{\Rpazo}{\pazocal{R}}

\newcommand{\Opazo}{\pazocal{O}}

\newcommand{\T}{\mathcal{T}}

\newcommand{\Lcal}{\mathcal{L}}

\newcommand{\Pcal}{\mathcal{P}}

\newcommand{\Scal}{\mathcal{S}}

\newcommand{\BC}{\textnormal{BC}}
\newcommand{\MC}{\textnormal{MC}}

\newcommand{\Id}{\textnormal{Id}}
\newcommand{\D}{\textnormal{D}}

\newcommand{\Tan}{\textnormal{Tan}}

\newcommand{\supp}{\textnormal{supp}}

\newcommand{\Lip}{\textnormal{Lip}}
\newcommand{\AC}{\textnormal{AC}}

\newcommand{\loc}{\textnormal{loc}}
\newcommand{\Div}{\textnormal{div}}

\newcommand{\dist}{\textnormal{dist}}
\newcommand{\textbn}[1]{\textnormal{\textbf{#1}}}
\newcommand{\co}{\overline{\textnormal{co}} \hspace{0.05cm}}

\newcommand{\vb}{\boldsymbol{v}}

\newcommand{\Bpi}{\boldsymbol{\pi}}

\newcommand{\Bnu}{\boldsymbol{\nu}}

\newcommand{\Bmu}{\boldsymbol{\mu}}

\newcommand{\INTDom}[3]{\int_{#2} #1 \textnormal{d} #3}
\newcommand{\INTSeg}[4]{\int_{#3}^{#4} #1 \textnormal{d} #2}
\newcommand{\NormL}[3]{\parallel \hspace{-0.1cm} #1 \hspace{-0.1cm} \parallel _ {L^{#2}(#3)}}
\newcommand{\NormC}[3]{\left\| #1  \right\| _ {C^{#2}(#3)}}
\newcommand{\Norm}[1]{\parallel \hspace{-0.1cm} #1 \hspace{-0.1cm} \parallel}

\newcommand{\tderv}[2]{\tfrac{\textnormal{d} #1}{ \textnormal{d} #2}}

\newtheorem{Def}{Definition}[section]
\newtheorem{thm}[Def]{Theorem}
\newtheorem{prop}[Def]{Proposition}
\newtheorem{rmk}[Def]{Remark}
\newtheorem{lem}[Def]{Lemma}
\newtheorem{cor}[Def]{Corollary}

\newtheorem{taggedhypx}{Hypotheses}
\newenvironment{taggedhyp}[1]
    {\renewcommand\thetaggedhypx{#1}\taggedhypx}
    {\endtaggedhypx}

\title{Necessary Optimality Conditions for Optimal Control Problems in Wasserstein Spaces}
\author{Benoît Bonnet\footnote{CNRS,  IMJ-PRG,  UMR  7586,  Sorbonne  Université, 4  place  Jussieu,  75252  Paris,  France. \hfill \hspace{3.5cm} \textit{E-mail}: \texttt{benoit.bonnet@imj-prg.fr} (Corresponding author)} , Hélène Frankowska\footnote{CNRS,  IMJ-PRG,  UMR  7586,  Sorbonne  Université,  4  place  Jussieu,  75252  Paris,  France. \hfill \hspace{3.5cm} \textit{E-mail}: \texttt{helene.frankowska@imj-prg.fr}}}
\date{\today}

\begin{document}

\maketitle

\begin{abstract}
In this article, we derive first-order necessary optimality conditions for a constrained optimal control problem formulated in the Wasserstein space of probability measures. To this end, we introduce a new notion of localised metric subdifferential for compactly supported probability measures, and investigate the intrinsic linearised Cauchy problems associated to non-local continuity equations. In particular, we show that when the velocity perturbations belong to the tangent cone to the convexification of the set of admissible velocities, the solutions of these linearised problems are tangent to the solution set of the corresponding continuity inclusion. We then make use of these novel concepts to provide a synthetic and geometric proof of the celebrated Pontryagin Maximum Principle for an optimal control problem with inequality final-point constraints. In addition, we propose sufficient conditions ensuring the normality of the maximum principle.
\end{abstract}

{\footnotesize

\vspace{0.25cm}

\textbf{MSC2020 Subject Classification :} 30L99, 34K09, 49J53, 49K21, 49Q22, 58E25
}


\section{Introduction}

In recent years, the mathematical analysis of collective behaviours in large systems of interacting agents has been the object of a growing attention from several mathematical communities. The corresponding class of so-called \textit{multi-agent systems} appears in a wide number of applications, including e.g. the modelling of autonomous vehicle ensembles \cite{Bullo2009}, swarms and flocking structures in the animal kingdom \cite{CS2}, opinion dynamics on networks \cite{Bellomo2013} and pedestrian flows \cite{CPT}. While these structures are frequently represented first-handedly as large dynamical systems formulated in finite-dimensional spaces -- e.g. by means of families of ordinary differential equations or time-dependent graphs --, such discrete formulations often give rise to intractable problems. Indeed, the main difficulty inherent to the study of multi-agent systems is their very large dimensionality, which causes severe challenges when trying to compute solutions using numerical methods or to describe agent-to-agent interactions in a meaningful way. These aspects are particularly relevant in \textit{control-theoretic} settings, where a policy maker aims at synthesising a control law which is as generic and global as possible in order to stir a system towards a desired goal. This type of control problems arise very frequently in the applied fields which were previously mentioned, see e.g. \cite{Albi2016,AlbiPareschiZanella,ControlKCS} and references therein.

In this context, the first step in order to circumvent these high-dimension related issues is to identify adequate infinite-dimensional approximations of the class of multi-agent systems at hand, usually in the form of a family of partial differential equations. One subsequently aims at synthesising a control law directly at the level of the PDE, and finally at returning to the original problem by using the ``infinite-dimensional'' control signal to stir the original finite-dimensional system (see e.g. the introduction of \cite{LipReg} for ampler details on this general scheme). To this day, the best identified framework to conduct such a program is that of \textit{mean-field approximations}. In this setting, the discrete agent trajectories are replaced by a curve of measures $\mu(\cdot)$, whose time-evolution is described by means of a \textit{non-local continuity equation} (see Section \ref{subsection:WassInc} below) of the form
\begin{equation}
\label{eq:Intro_CE}
\partial_t \mu(t) + \Div \big( v(t,\mu(t)) \mu(t) \big) = 0.
\end{equation}
These equations appear very naturally when performing infinite-dimensional approximations of deterministic particle systems driven by ODEs, and are usually studied in the space of probability measures endowed with the \textit{Wasserstein metrics} of optimal transport (see Section \ref{subsection:OptTransport} below). For a modern introduction to this topic, we point to \cite{golse} as well as to the survey \cite{Choi2014}.

In this article, we investigate necessary optimality conditions for the \textit{mean-field optimal control problems} of Bolza type
\begin{equation*}
(\Ppazo) ~~ \left\{
\begin{aligned}
\min_{u(\cdot) \in \U} & \left[ \INTSeg{L(t,\mu(t),u(t))}{t}{0}{T} + \varphi(\mu(T)) \right] \\
\text{s.t.} ~ & \left\{
\begin{aligned}
& \partial_t \mu(t) + \Div \big( v(t,\mu(t),u(t)) \mu(t) \big) = 0, \\
& \, \mu(0) = \mu^0, \\
&  \mu(T) \in \Qpazo_T.
\end{aligned}
\right.
\end{aligned}
\right.
\end{equation*}
The minimisation in problem $(\Ppazo)$ is taken over the set of admissible \textit{open-loop} controls $\U := \{ t \in [0,T] \mapsto u(t) \in U ~\text{s.t. $u(\cdot)$ is $\Lcal^1$-measurable} \}$, where $(U,d_U)$ is a compact metric space and $\Lcal^1$ denotes the standard one-dimensional Lebesgue measure. The state $\mu(\cdot)$ of the system is stirred by the controlled non-local velocity field $v : [0,T] \times \Pcal_c(\R^d) \times U \times \R^d \rightarrow \R^d$, while $\mu^0 \in \Pcal_c(\R^d)$ denotes a fixed initial datum and $\Qpazo_T \subset \Pcal_2(\R^d)$ is a closed set which represents final-point constraints. As customary in the modelling of optimal control problems, the cost is split into a final cost $\varphi : \Pcal_c(\R^d) \rightarrow \R$ and a distributed running cost $L : [0,T] \times \Pcal_c(\R^d) \times U \rightarrow \R$. Observe that in $(\Ppazo)$, the state variable $\mu(\cdot)$ is modelled as a curve of \textit{compactly supported} measures. This comes from the fact that the afore-described  applications of multi-agent systems usually involve compactly supported initial data, and that classical stability estimates (see e.g. Theorem \ref{thm:ExistenceWass} below) ensure that curves of measure solving non-local continuity equations starting from compactly supported data remain compactly supported at all subsequent times. 

Optimal control problems in Wasserstein spaces of the form $(\Ppazo)$ have been extensively studied in the past few years, mostly in the absence of final-point or running constraints. The first contributions on this topic \cite{Fornasier2019,FornasierPR2014,Fornasier2014} were concerned with the rigorous convergence of discrete multi-agent optimal controls towards their infinite-dimensional counterparts (see also the more recent contributions \cite{Cavagnari2021,Fornasier2019}), which provides a constructive way of recovering the existence of optimal trajectories. In parallel to these investigations, several attempts were made at deriving \textit{necessary optimality conditions} for mean-field optimal control problems. The first result of this kind was proposed in \cite{MFPMP}, where the authors proved a generalisation of the celebrated \textit{Pontryagin Maximum Principle} (``PMP'' in the sequel) for a multi-scale version of $(\Ppazo)$. Therein, a family of open-loop controls  acts on a finite-dimensional subset of \textit{leaders}, which allows to recover a PMP in Wasserstein spaces by passing to the limit in the finite-dimensional maximum principles as the number of agents goes to infinity. The first general derivation of the PMP for problem $(\Ppazo)$ was obtained by the first author in \cite{PMPWass}, and is based on a careful adaptation of the strategy of \textit{needle-variations} to the abstract geometric structure of Wasserstein spaces. This result was further extended in \cite{PMPWassConst} to problems with general final-point and running state constraints. In the latter contribution, the proof strategy combines a finite-dimensional non-smooth multipliers rule and \textit{outer-approximations} of optimal trajectories by countable families of curves generated using needle-variations. We also mention the earlier work \cite{Pogodaev2016} in which the author independently derived an alternative version of the maximum principle for a very particular instance of problem $(\Ppazo)$. This result was subsequently extended in the recent paper \cite{Pogodaev2020}, in which a full PMP in the spirit of \cite{PMPWass} is derived for unconstrained \textit{impulsive} control problems, by means of discrete approximations combined with Ekeland's principle. We likewise point out that a completely different approach to Pontryagin optimality conditions for mean-field optimal control problems -- formulated in terms of McKean-Vlasov dynamics and stochastic differential equations -- was independently developed in \cite{Carmona2015} (see also the recent monograph \cite[Vol.I Chapter 6]{Carmona2018}).

In this manuscript, we revisit the necessary optimality conditions for problem $(\Ppazo)$ by proposing a geometric proof of the maximum principle (see Theorem \ref{thm:PMPMayer} and Theorem \ref{thm:PMPBolza} below). The latter is based on a general methodology which was first developed by the second author in \cite{Frankowska1987} for Mayer problems, and that can be summarised as follows. When the running cost $L(\cdot,\cdot,\cdot)$ is equal to zero, the generalised Fermat rule states that the directional derivatives of the final cost $\varphi(\cdot)$ taken at points of local minimum are non-negative when evaluated along directions that are tangent to the set of admissible trajectories. Usually, due to the non-linearity of the dynamics, it is not possible to exhaustively describe the set of all such tangents. However in order to derive the PMP, it is sufficient to identify subsets of tangent directions which are informative enough. As it can be expected from the classical theory of optimisation, the set of trajectories of the linearised control system satisfying linearised constraints are tangent to the set of all admissible trajectories, provided some additional constraint qualifications are satisfied. While classical linearisation methods for control systems usually involve derivatives of the dynamics with respect to both state and control, these latter only describe a small subset of tangent directions and lead to the so-called \textit{weak maximum principle}. On the other hand, replacing the control system by a \textit{relaxed differential inclusion} does not change the set of tangents to its trajectories, even though it drastically enlarges the dynamics itself (since the right-hand side becomes the convex hull of the admissible velocities of the original control system). Then, linearisations  of this  ``larger'' set of dynamics obtained by taking adequate derivatives of the resulting set-valued velocity mapping allow to describe richer subsets of tangents. Note next that the Fermat rule also implies that the set of curves which are tangent to admissible trajectories has an empty intersection with the  the set of all directions in which the cost is strictly decreasing.  At this point, one can recover an explicit dual description of this fact in the form of an inequality by applying a separation theorem. The final step in this procedure is then to introduce a suitably chosen \textit{costate variable} satisfying the adjoint dynamics, which is such that the maximisation condition of the PMP holds for almost-every time. Besides, it is well known that Bolza problems can be equivalently written in Mayer form -- i.e. with $L = 0$ in $(\Ppazo)$ --, which means that the approach we just described is also valid for optimal control problems with non-trivial running costs. When the final-point constraints are described by families of functional inequalities and the set of tangents to the latter has empty interior, it follows that the abnormal version of the PMP holds true,  i.e. the adjoint state does not depend on the cost of the problem. On the other hand when the tangent set to the final-point constraints has non-empty interior and when the constraints are qualified in some sense, then the normal version of the PMP holds true, and the investigation of optimal curves is greatly simplified.

The approach that we described in the previous paragraph is general enough to continue with second- (and higher-) order optimality conditions, by taking second-order linearisations of control systems and constraints, as well as second-order derivatives of the costs, even in the presence of additional running state constraints. Another interesting feature of this strategy is that it allows to precisely discriminate between the normal and abnormal versions of the PMP, which is generically not true for proofs relying on outer-approximations of optimal trajectories. This general methodology has been successfully applied to optimal control problems formulated on very broad classes of dynamics, including e.g. classical ODEs in finite-dimensional spaces \cite{Frankowska1987}, semi-linear evolution equations in general Banach spaces \cite{Frankowska1990,FrankowskaMM2018} and stochastic differential equations \cite{FrankowskaZZ2018}. In addition to its holistic nature, this procedure also presents the advantage of being readily transposable to the study of \textit{second-order} necessary optimality conditions for optimal control problems, see e.g. \cite{FrankowskaL2020,Frankowska2017,Frankowska2018} and references therein. We point more specifically to the recent publication \cite{FrankowskaO2020} on this topic, in which an optimal control problem involving final-point equality and inequality constraint was investigated using metric inverse mapping theorems on the space of controls.

However, in the context of mean-field optimal control problems written in the general form $(\Ppazo)$, the implementation of the afore-described program is much more delicate. Indeed, the space of probability measures is at best a metric space when endowed with a Wasserstein metric, and a number of tools and concepts of set-valued and non-smooth analysis need to be carefully adapted to this setting in order to study necessary optimality conditions. Concerning differential-theoretic aspects, it has been hinted in the pioneering works \cite{McCann1997,Otto2001} and further installed in \cite{AGS,Villani1} that the Wasserstein space $(\Pcal_2(\R^d),W_2)$ can be endowed with a \textit{pseudo-Riemannian} structure. More specifically, the latter provides an explicit characterisation of the exponential map and of the corresponding \textit{analytical tangent space} $\Tan_{\mu} \Pcal_2(\R^d)$ (see Section \ref{subsection:OptTransport} below), and a full theory of first- and second-order differential calculus has been developed in this framework, see \cite[Chapter 10]{AGS} and \cite{Ambrosio2008,Gangbo2019}. Regarding differential inclusions, the authors of the present manuscript have recently proposed in \cite{ContInc} a sound generalisation of these objects for continuity equations formulated in Wasserstein spaces, in keeping with the classical theory that is described e.g. in \cite[Chapter 10]{Aubin1990}. In particular contrary to other attempts aiming at generalising differential inclusions in this context (see e.g. \cite{CavagnariMP2018,Jimenez2020}), the notion of \textit{continuity inclusion} introduced in \cite{ContInc} ensures a one-to-one correspondence between the trajectories of control systems and the canonically associated set-valued dynamics (see Theorem \ref{thm:ControlInc} below).

\bigskip

The contributions of this manuscript can be summarised as follows. In Section \ref{subsection:LocalisedSubdiff}, we start by proposing a new definition of \textit{localised subdifferentials} and \textit{differentiability} (see Definition \ref{def:Subdiff} and Definition \ref{def:LocalDiff} below) for functionals defined over Wasserstein spaces, which are well-suited to the study of dynamical and variational problems involving compactly supported measures. In Proposition \ref{prop:StrongSubdiff}, we adapt to this new setting a result already known for absolutely continuous measures in the classical case (see e.g. \cite[Chapter 1, Proposition 4.2]{Dafermos2006}), which asserts that subdifferentials which belong to $\Tan_{\mu}\Pcal_2(\R^d)$ are in fact strong (see Definition \ref{def:Subdiff} below). We subsequently derive the expression of the intrinsic linearised Cauchy problem associated to a non-local continuity equation in Proposition \ref{prop:Linearisation_Cauchy} of Section \ref{subsection:Linearisation}, and study in Theorem \ref{thm:Variational_Inclusion} the \textit{variational linearisation} of continuity inclusions. Thereafter in Section \ref{subsection:Mayer}, we apply these differential and set-theoretic concepts to prove a Pontryagin Maximum Principle for a \textit{Mayer} version of problem $(\Ppazo)$ -- i.e. without running cost -- in Theorem \ref{thm:PMPMayer}, where the set of end-point constraints $\Qpazo_T$ is defined by \textit{functional inequalities}. We also point out that the proof of this result has been slightly improved compared to earlier implementations of this program e.g. in \cite{FrankowskaMM2018,FrankowskaO2020}, as the separation arguments are performed in adequate \textit{finite-dimensional} spaces, which circumvents a lot of potential technicalities related to the separation of subsets of infinite-dimensional vector spaces. A full PMP for the corresponding Bolza problem is then recovered in Theorem \ref{thm:PMPBolza} of Section \ref{subsection:Bolza} by means of a standard procedure that directly builds on the PMP for the Mayer problem.

Before moving on to the core of the manuscript, we would like to make some final remarks.
\begin{enumerate}
\item While our methodology would also allow to deal with end-point constraints involving \textit{functional equalities}, the corresponding proof strategy relies on a different class of inner-approximations, combined with technical metric inverse mapping arguments in the spirit of \cite{FrankowskaO2020}. For the sake of clarity, we shall present these results separately in a future publication.
\item Even though the constrained PMP derived in \cite{PMPWassConst} incorporates state constraints and takes into account more general final-point constraints, its proof suffers from the drawbacks of being less geometrically meaningful and highly technical, while relying on outer-approximations. In this regard, the approach described in this manuscript is more synthetic, and should allow for the derivation of the maximum principle in the presence of general constraints (see e.g. \cite{Frankowska2018,FrankowskaO2020}).
\item The study of mean-field control problems of the form $(\Ppazo)$ is also sensible when the admissible controls $u : [0,T] \times \R^d \rightarrow \R^d$ have a \textit{closed-loop} structure with respect to the space variable, see e.g. \cite{PMPWassConst,PMPWass,SparseJQMF,ControlKCS} and \cite{Carmona2018,Lasry2007} in connection with the theory of mean-field games. However in this case, derivatives with respect to the space variable of the closed-loop controls also appear in the linearised systems and in  the adjoint dynamics of the PMP, which requires more regularity assumptions and leads to heavier expressions while using the same geometric ideas as in the open-loop case.
\end{enumerate}

The structure of the article is the following. In Section \ref{section:Preliminaries}, we expose a list of preliminary notions of optimal transport theory and set-valued analysis. In particular, we recall some of the recent results of \cite{ContInc} about continuity inclusions. In Section \ref{section:NonSmoothWass}, we present our new definition of localised subdifferentials along with the intrinsic and variational linearisations of the set-valued counterparts of continuity equations. In Section \ref{section:PMP}, we  subsequently prove a PMP in Wasserstein spaces, first for a Mayer problem, and then for a Bolza problem. We finally present in Section \ref{appendix:Examples} examples of functionals which are locally differentiable in the sense of Section \ref{section:NonSmoothWass}, and provide the expressions of their gradients.


\section{Preliminaries}
\label{section:Preliminaries}

In this section, we recall several notions pertaining to optimal transport theory, set-valued analysis and continuity inclusions in Wasserstein spaces, all of which will be needed to state and prove the main results of Section \ref{section:NonSmoothWass} and Section \ref{section:PMP}.


\subsection{Analysis in measures spaces and optimal transport}
\label{subsection:OptTransport}

We recollect here some preliminary tools of analysis in measure spaces and optimal transport theory, for which we refer the reader to the monographs \cite{AmbrosioFuscoPallara} and \cite{AGS,OTAM,Villani1} respectively.

Given two complete separable metric spaces $(\Scal,d_{\Scal})$ and $(\T,d_{\T})$, we denote by $\overline{\Omega}$ the closure of a subset $\Omega \subset \Scal$ and by  $\partial  \Omega$ its topological boundary. A mapping $\phi : \Scal \rightarrow \T$ is said to be bounded if for some (and thus all) $\tau_0 \in \T$, it holds that $\sup_{s \in \Scal} d_{\T}(\phi(s),\tau_0) < +\infty$. We will use the notation $C^0(\Scal,\T)$ (resp. $C^0_b(\Scal,\T)$) for the space of continuous (resp. continuous and bounded) functions from $\Scal$ into $\T$, as well as $\AC([0,T],\Scal)$ for the space of absolutely continuous curves from $[0,T] \subset \R_+$ into $\Scal$. In the sequel, $\Lip(\phi(\cdot) \, ; \Omega)$ stands for the Lipschitz constant of a map $\phi : \Scal \rightarrow \T$ over $\Omega \subset \Scal$. Given a separable Banach space $(X,\Norm{\cdot}_X)$ and $p \in [1,+\infty)$, we denote by $L^p([0,T],X)$ the Banach space of \textit{$p$-integrable maps} from $[0,T] \subset \R_+$ into $X$ in the sense of Bochner (see e.g. \cite{DiestelUhl}), where $[0,T]$ is endowed with the standard one-dimensional Lebesgue measure $\Lcal^1$.

Let $\Pcal(\R^d)$ be the space of \textit{Borel probability measures} defined over $\R^d$ endowed with the \textit{narrow topology}, that is the topology induced by the weak-$^*$ convergence of measures defined by
\begin{equation}
\label{eq:Narrow}
\mu_n ~ \underset{n \rightarrow +\infty}{\, \rightharpoonup^*}~ \mu \quad \text{if and only if} \quad \INTDom{\phi(x)}{\R^d}{\mu_n(x)} ~\underset{n \rightarrow +\infty}{\longrightarrow}~ \INTDom{\phi(x)}{\R^d}{\mu(x)} \quad \text{for every $\phi \in C^0_b(\R^d,\R)$}.
\end{equation}
Given $p \in [1,+\infty)$, we define the \textit{momentum of order $p$} of an element $\mu \in \Pcal(\R^d)$ by
\begin{equation*}
\M_p(\mu) := \bigg( \INTDom{|x|^p}{\R^d}{\mu(x)} \bigg)^{1/p},
\end{equation*}
as well as its \textit{support} as the closed subset of $\R^d$
\begin{equation*}
\supp(\mu) := \Big\{ x \in \R^d ~\text{s.t.}~ \mu(\Npazo_x) > 0 ~~ \text{for every neighbourhood $\Npazo_x$ of $x$} \Big\}.
\end{equation*}
We will henceforth denote by $\Pcal_p(\R^d)$ and $\Pcal_c(\R^d)$ the subsets of Borel probability measures with finite momentum of order $p$ and compact support respectively. In what follows, we will frequently use the \textit{$R$-fattening} of the support of a measure $\mu \in \Pcal_c(\R^d)$, defined by
\begin{equation}
\label{eq:Fattening}
B_{\mu}(R) ~ := \bigcup_{x \, \in \, \supp(\mu)} \hspace{-0.2cm} B(x,R), \vspace{-0.15cm}
\end{equation}
where $B(x,R) \subset \R^d$ denotes the closed ball of center $x \in \R^d$ with radius $R > 0$.

\begin{Def}[Pushforward of a measure and transport plans]
Given a Borel map $f : \R^d \rightarrow \R^d$ and an element $\mu \in \Pcal(\R^d)$, the \textnormal{pushforward} $f_{\#} \mu \in \Pcal(\R^d)$ of $\mu$ through $f(\cdot)$ is defined as
\begin{equation*}
f_{\#} \mu(B) := \mu(f^{-1}(B)),
\end{equation*}
for any Borel set $B \subset \R^d$. Given two elements $\mu,\nu \in \Pcal(\R^d)$, the set $\Gamma(\mu,\nu)$ of \textnormal{transport plans} between $\mu$ and $\nu$ is defined as the subset of all $\gamma \in \Pcal(\R^{2d})$ such that $\pi^1_{\#} \gamma = \mu$ and $\pi^2_{\#} \gamma = \nu$, where the maps $\pi^1,\pi^2 : \R^d \times \R^d \rightarrow \R^d$ stand for the projection operations onto the first and second factors.
\end{Def}

We recall below the definition and some of the known properties of the \textit{Wasserstein spaces} of optimal transport theory (see e.g. \cite{AGS,OTAM,Villani1}).

\begin{Def}[Wasserstein spaces]
Given $p \in [1,+\infty)$ and $\mu,\nu \in \Pcal_p(\R^d)$, the \textnormal{Wasserstein distance of order $p$} between $\mu$ and $\nu$ is defined by
\begin{equation*}
W_p(\mu,\nu) \, := \min_{\gamma \in \Gamma(\mu,\nu)} \bigg( \INTDom{|x-y|^p}{\R^{2d}}{\gamma(x,y)} \bigg)^{1/p},
\end{equation*}
and we denote by $\Gamma_o(\mu,\nu)$ the set of \textnormal{optimal transport plans} for which this minimum is attained. The \textnormal{Wasserstein space of order $p$} is then defined as the metric space $(\Pcal_p(\R^d),W_p)$ of probability measures with finite momentum of order $p$ endowed with the $W_p$-distance.
\end{Def}

It is a well-known fact in optimal transport theory that the sets $\Gamma_o(\mu,\nu)$ are non-empty for every $p \in [1,+\infty)$ and any $\mu,\nu \in \Pcal_p(\R^d)$. This is a direct consequence of the narrow closedness of $\Gamma(\mu,\nu)$, together with the coercivity of the map
\begin{equation*}
\gamma \in \Pcal_p(\R^{2d}) \mapsto \INTDom{|x-y|^p}{\R^{2d}}{\gamma(x,y)} \in \R_+.
\end{equation*}
Moreover, the space $(\Pcal_p(\R^d),W_p)$ is a complete separable metric space, and the topology induced by the Wasserstein distance metrises the narrow topology \eqref{eq:Narrow} restricted to $\Pcal_p(\R^d)$, in the sense that
\begin{equation*}
W_p(\mu_n,\mu) ~\underset{n \rightarrow +\infty}{\longrightarrow}~ 0 \qquad \text{if and only if} \qquad
\left\{
\begin{aligned}
\mu_n ~ & \underset{n \rightarrow +\infty}{\, \rightharpoonup^*}~ \mu, \\
\INTDom{|x|^p}{\R^d}{\mu_n(x)} & \underset{n \rightarrow +\infty}{\longrightarrow} \INTDom{|x|^p}{\R^d}{\mu(x)},
\end{aligned}
\right.
\end{equation*}
for every sequence $(\mu_n) \subset \Pcal_p(\R^d)$ and $\mu \in \Pcal_p(\R^d)$. In the sequel, we will frequently consider the (non-complete) metric space $(\Pcal_c(\R^d),W_p)$ of compactly supported measures, seen as a subset of the Wasserstein space $(\Pcal_p(\R^d),W_p)$, where $p \in [1,+\infty)$ will depend on the context. 

In the particular case where $p = 2$, it is known that $(\Pcal_2(\R^d),W_2)$ can also be endowed with a \textit{pseudo-Riemannian structure}. Given an element $\mu \in \Pcal_2(\R^d)$, the \textit{analytical tangent space} to $\Pcal_2(\R^d)$ at $\mu$ is defined in this context as
\begin{equation*}
\Tan_{\mu} \Pcal_2(\R^d) := \overline{\big\{ \nabla \xi(\cdot) ~\text{s.t.}~ \xi \in C^{\infty}_c(\R^d)  \big\}}^{L^2(\mu)}.
\end{equation*}
We point the reader to \cite[Chapter 8]{AGS} for a contextual introduction of this geometric object, as well as for a thorough analysis of its interplay with absolutely continuous curves of measures and continuity equations (see also Section \ref{subsection:WassInc} below).

We end this series of prerequisites by recalling a variant of the \textit{disintegration theorem} in the context of optimal transport, for which we refer e.g. to \cite[Theorem 5.3.1]{AGS}.

\begin{thm}[Disintegration theorem]
\label{thm:Disintegration}
Given two elements $\mu,\nu \in \Pcal_p(\R^d)$ and $\gamma \in \Gamma(\mu,\nu)$, there exists a $\mu$-almost uniquely determined family of measures $\{ \gamma_x \}_{x \in \R^d} \subset \Pcal_p(\R^d)$ called the \textnormal{disintegration} of $\gamma$ onto its first marginal $\pi^1_{\#} \gamma = \mu$, such that
\begin{equation*}
\INTDom{\xi(x,y)}{\R^{2d}}{\gamma(x,y)} = \INTDom{\INTDom{\xi(x,y)}{\R^d}{\gamma_x(y)}}{\R^d}{\mu(x)},
\end{equation*}
for every map $\xi \in L^1(\R^{2d},\R;\gamma)$.
\end{thm}

Reciprocally given $\mu \in \Pcal_c(\R^d)$ and a $\mu$-measurable family of measures $\{\gamma_x\}_{x \in \R^d} \subset \Pcal_c(\R^d)$, we shall sometimes write $\gamma := \INTDom{\gamma_x}{\R^d}{\mu(x)}$ to denote the unique element of $\Pcal_c(\R^{2d})$ satisfying 
\begin{equation*}
\INTDom{\xi(x,y)}{\R^{2d}}{\gamma(x,y)} = \INTDom{\INTDom{\xi(x,y)}{\R^d}{\gamma_x(y)}}{\R^d}{\mu(x)}
\end{equation*}
for every $\xi \in C^0_b(\R^{2d},\R)$. 


\subsection{Elements of set-valued analysis}

In this section, we recall several notions and classical results of set-valued and non-smooth analysis, for which we refer to the comprehensive monographs \cite{Aubin1984,Aubin1990}.

Let $(\Scal,d_{\Scal})$ be a complete separable metric space and $(X,\Norm{\cdot}_X)$ be a separable Banach space. We denote by $\B_{\Scal}(s,\beta)$ the closed ball of radius $\beta > 0$ centred at $s \in \Scal$, and by $\B_X$ the closed unit ball in $X$ centred at $0$. Given a subset $B \subset X$, we define its \textit{closed convex hull} $\co(B)$ as the closure of
\begin{equation*}
\textnormal{co}(B) := \bigg\{ \mathsmaller{\sum}\nolimits_{i=1}^N \alpha_i b_i ~\text{s.t.}~ N\geq 1,~ b_i \in B,~ \alpha_i \geq 0 ~\text{for $i \in \{1,\dots,N\}$ and}~ \mathsmaller{\sum}\nolimits_{i=1}^N \alpha_i = 1 \bigg\}  \subset X, 
\end{equation*}
and consider the \textit{distance function} to $B$ defined as $x \in X \mapsto \dist_X(x \, ; B) := \inf_{b \in B} \Norm{x-b}_X$. An application $\F : \Scal \rightrightarrows X$ is called a \textit{set-valued map} -- or a \textit{multifunction} -- from $\Scal$ into $X$ if $\F(s) \subset X$ for any $s \in \Scal$, and we denote by $D(\F) := \{ s \in \Scal ~\text{s.t.}~ \F(s) \neq \emptyset \}$ its \textit{effective domain}. A set-valued map $\F(\cdot)$ is said to have closed (resp. convex) images if the sets $\F(s)$ are closed (resp. convex) for any $s \in D(\F)$. In the following definitions, we recall the concepts of \textit{measurability} and \textit{Lipschitz regularity} for set-valued maps.

\begin{Def}[Measurable multifunctions]
Let $(\Scal,\Omega,\varpi)$ be a measure space defined over $\Scal$ and $\F : \Scal \rightrightarrows X$ be a set-valued map. Then, $\F(\cdot)$ is said to be \textnormal{$\varpi$-measurable} if
\begin{equation*}
\F^{-1}(\Opazo) := \Big\{ s \in \Scal ~\text{s.t.}~ \F(s) \cap \Opazo \neq \emptyset \Big\} \in \Omega,
\end{equation*}
for any open set $\Opazo \subset X$.
\end{Def}

\begin{Def}[Lipschitz continuity of multifunctions]
A set-valued map $\F : \Scal \rightrightarrows X$ is \textnormal{Lipschitz regular} around $s \in D(\F)$ with constant $L > 0$ if there exists a neighbourhood $\Npazo_s \subset \Scal$ of $s$ such that
\begin{equation*}
\F(s_1) \subset \F(s_2) + L \,d_{\Scal}(s_1,s_2) \B_X,
\end{equation*}
for any $s_1,s_2 \in \Npazo_s$.
\end{Def}

We end this primer in set-valued analysis by recalling the definition of \textit{tangent cone} to a closed convex subset of a Banach space, and provide one of its equivalent characterisations.

\begin{Def}[Tangent cone to a convex set]
\label{def:AdjacentConeBanach}
Let $K \subset X$ be a closed convex set and $x \in K$. The \textnormal{tangent cone} $T_K(x)$ to $K$ at $x$ is then defined by
\begin{equation*}
T_K(x) := \overline{\bigcup_{\lambda > 0} \lambda(K-x)}.
\end{equation*}
Moreover, the set $T_K(x)$ can be alternatively characterised as
\begin{equation}
\label{eq:AlternativeCone}
\begin{aligned}
T_K(x) & = \Big\{ w \in X ~\text{s.t.}~ \dist_X(x + \epsilon w ; K) = o(\epsilon) ~~ \text{for any $\epsilon > 0$} \Big\} \\
& = \bigg\{ w \in X ~\text{s.t. for every $\epsilon_n \underset{n \rightarrow +\infty}{\longrightarrow} 0^+$, there exist $w_n \underset{n \rightarrow +\infty}{\longrightarrow} w$ such that} \\
& \hspace{7.75cm} x + \epsilon_n w_n \in K ~~ \text{for all $n \geq 1$} \bigg\}.
\end{aligned}
\end{equation}
as a consequence e.g. of \cite[Proposition 4.2.1]{Aubin1990}.
\end{Def}


\subsection{Continuity equations and inclusions in Wasserstein spaces}
\label{subsection:WassInc}

In this section, we recollect notions pertaining to \textit{continuity equations} and \textit{continuity inclusions} in the space of measures. Given a time horizon $T > 0$, a curve $\mu : [0,T] \rightarrow \Pcal(\R^d)$ and a velocity-field $v : [0,T] \times \R^d \rightarrow \R^d$, we say that $(\mu(\cdot),v(\cdot))$ is a trajectory-velocity pair of the continuity equation
\begin{equation}
\label{eq:CE}
\partial_t \mu(t) + \Div \big( v(t) \mu(t) \big) = 0,
\end{equation}
if the following distributional equality
\begin{equation*}
\INTSeg{\INTDom{\Big( \partial_t \xi(t,x) + \big\langle \nabla_x \xi(t,x) , v(t,x) \big\rangle\Big)}{\R^d}{\mu(t)(x)}}{t}{0}{T} = 0,
\end{equation*}
holds for any test function $\xi \in C^{\infty}_c((0,T) \times \R^d)$. In what follows, we will mainly consider continuity equations driven by \textit{Carathéodory} velocity-fields, namely maps such that $t \in [0,T] \mapsto v(t,x) \in \R^d$ is $\Lcal^1$-measurable for any $x \in \R^d$ and $x \in \R^d \mapsto v(t,x) \in \R^d$ is continuous for $\Lcal^1$-almost every $t \in [0,T]$.

In the recent work \cite{ContInc}, we introduced a \textit{set-valued} extension of \eqref{eq:CE}, in the case where the velocity-field is \textit{non-local} i.e. when it also depends on the whole measure $\mu(t)$ at all times $t \in [0,T]$. Based on the identification of non-local velocity-fields
\begin{equation*}
(t,\mu,x) \in [0,T] \times \Pcal_c(\R^d) \times \R^d \mapsto v(t,\mu,x) \in \R^d,
\end{equation*}
with vector-field valued maps
\begin{equation*}
(t,\mu) \in [0,T] \times \Pcal_c(\R^d) \mapsto v(t,\mu,\cdot) \in C^0(\R^d,\R^d),
\end{equation*}
we proposed the following notion of differential inclusion for continuity equations.

\begin{Def}[Continuity inclusions in measure spaces]
\label{def:WassInc}
Let $V : [0,T] \times \Pcal_c(\R^d) \rightrightarrows C^0(\R^d,\R^d)$ be a set-valued map. We say that a curve of measures $\mu(\cdot)$ solves the \textnormal{continuity inclusion}
\begin{equation}
\label{eq:WassIncDef}
\partial_t \mu(t) \in - \Div \Big( V(t,\mu(t)) \mu(t) \Big),
\end{equation}
if there exists an $\Lcal^1$-measurable selection $t \in [0,T] \mapsto \vb(t) \in V(t,\mu(t))$ such that the \textnormal{trajectory-selection pair} $(\mu(\cdot),\vb(\cdot))$ solves the continuity equation
\begin{equation}
\label{eq:ContinuityEqInc}
\partial_t \mu(t) + \Div \big( \vb(t) \mu(t) \big) = 0,
\end{equation}
in the sense of distributions.
\end{Def}

In what follows, we introduce technical notions pertaining to \textit{compact restrictions} of multifunctions with values in $C^0(\R^d,\R^d)$, and proceed by stating our working assumptions. We will henceforth denote by $\Norm{\cdot}_1 := \NormL{\cdot}{1}{[0,T],\R}$ the $L^1$-norm of a real-valued functional defined over $[0,T]$.

\begin{Def}[Restriction of multifunctions to balls]
\label{def:Compactification}
Given a set-valued map $V : [0,T] \times \Pcal_c(\R^d) \rightrightarrows C^0(\R^d,\R^d)$ and $R > 0$, we define its \textnormal{compact restriction} to $K := B(0,R)$ by
\begin{equation*}
V_K(t,\mu) := \Big\{ v_{\vert K} \in C^0(K,\R^d) ~\text{s.t.}~ v \in V(t,\mu) \Big\},
\end{equation*}
for all $(t,\mu) \in [0,T] \times \Pcal_c(\R^d)$. Here, $v_{\vert K}$ denotes the restriction to $K$ of the map $v \in C^0(\R^d,\R^d)$. Conversely, observe that any function $v \in C^0(K,\R^d)$ can be extended to $\tilde{v} \in C^0(\R^d,\R^d)$ by defining $\tilde{v}(x) := v(\pi_K(x))$ for all $x \in \R^d$, where $\pi_K : \R^d \rightarrow K$ denotes the projection onto $K$.

We also define the \textnormal{closed convex hull} $\co V(t,\mu)$ of $V(t,\mu) \subset C^0(\R^d,\R^d)$ as
\begin{equation}
\label{eq:ConvbarDef}
\co V(t,\mu) := \Big\{ v \in C^0(\R^d,\R^d) ~\text{s.t.}~ v_{\vert K} \in \co V_K(t,\mu) ~\text{for any $K := B(0,R)$ with $R>0$}\Big\},
\end{equation}
where the closed convex hull $\co V_K(t,\mu)$ is taken in the Banach space $\big(C^0(K,\R^d),\NormC{\cdot}{0}{K,\R^d} \big)$.
\end{Def}

Throughout the remainder of this section, we fix a time horizon $T > 0$ and a real $p \in [1,+\infty)$.

\begin{taggedhyp}{\textbn{(DI)}}
\label{hyp:DI}
For every $R > 0$, assume that the following holds with $K := B(0,R)$.

\begin{enumerate}
\item[$(i)$] For any $\mu \in \Pcal_c(\R^d)$, the set-valued map $t \in [0,T] \rightrightarrows V_K(t,\mu)$ is $\Lcal^1$-measurable with closed non-empty images in $C^0(K,\R^d)$.
\item[$(ii)$] There exists a map $m(\cdot) \in L^1([0,T],\R_+)$ such that for $\Lcal^1$-almost every $t \in [0,T]$, any $\mu \in \Pcal_c(\R^d)$, every $v \in V(t,\mu)$ and all $x \in \R^d$, it holds
\begin{equation*}
|v(x)| \leq m(t) \Big(1+|x|+\M_1(\mu) \Big).
\end{equation*}
\item[$(iii)$] There exists a map $l_K(\cdot) \in L^1([0,T],\R_+)$ such that for $\Lcal^1$-almost every $t \in [0,T]$, any $\mu \in \Pcal(K)$ and every $v \in V(t,\mu)$, it holds
\begin{equation*}
\Lip \big(v(\cdot) \,;K \big) \leq l_K(t).
\end{equation*}
\item[$(iv)$] There exists a map $L_K(\cdot) \in L^1([0,T],\R_+)$ such that for $\Lcal^1$-almost every $t \in [0,T]$ and any $\mu,\nu \in \Pcal(K)$, it holds
\begin{equation*}
V_K(t,\nu) \subset V_K(t,\mu) + L_K(t) W_p(\mu,\nu) \B_{C^0(K,\R^d)}.
\end{equation*}
\end{enumerate}
\end{taggedhyp}

\begin{rmk}[Non-local continuity equations as a particular case of continuity inclusions]
\label{rmk:NonLocalCE}
Suppose that the multifunction $V : [0,T] \times \Pcal_c(\R^d) \rightrightarrows C^0(\R^d,\R^d)$ is \textnormal{single-valued} for all $(t,\mu) \in [0,T] \times \Pcal_c(\R^d)$. Then, solutions of \eqref{eq:WassIncDef} coincide with those of the non-local continuity equation
\begin{equation}
\label{eq:NonLocalCE}
\partial_t \mu(t) + \Div \big( v(t,\mu(t)) \mu(t) \big) = 0,
\end{equation}
where $v(t,\mu,x) := V(t,\mu)(x)$ for all $(t,\mu,x) \in [0,T] \times \Pcal_c(\R^d) \times \R^d$.
\end{rmk}

In the scenario described in Remark \ref{rmk:NonLocalCE}, hypotheses \ref{hyp:DI} become a localised and time-dependent generalisation of that of \cite[Section 1]{Pedestrian} in the spirit of Carathéodory ODEs. We state these assumptions separately for the sake of clarity, as they will be frequently used in the sequel.

\begin{taggedhyp}{\textbn{(CE)}}
\label{hyp:CE}
For every $R > 0$, suppose that the following holds with $K := B(0,R)$.
\begin{enumerate}
\item[$(i)$] The application $t \in [0,T] \mapsto v(t,\mu,x) \in \R^d$ is $\Lcal^1$-measurable for any $(\mu,x) \in \Pcal_c(\R^d) \times \R^d$. Moreover, there exists a map $m(\cdot) \in L^1([0,T],\R_+)$ such that
\begin{equation*}
|v(t,\mu,x)| \leq m(t) \Big( 1 + |x| + \M_1(\mu) \Big),
\end{equation*}
for $\Lcal^1$-almost every $t \in [0,T]$ and any $(\mu,x) \in \Pcal_c(\R^d) \times \R^d$.
\item[$(ii)$] There exist two maps $l_K(\cdot),L_K(\cdot) \in L^1([0,T],\R_+)$ such that for $\Lcal^1$-almost every $t \in [0,T]$,  any $\mu,\nu \in \Pcal(K)$ and all $x,y \in K$, it holds
\begin{equation*}
\big| v(t,\mu,x) - v(t,\mu,y) \big| \leq l_K(t) |x-y| \qquad \text{and} \qquad \big| v(t,\mu,x) - v(t,\nu,x) \big| \leq L_K(t) W_p(\mu,\nu).
\end{equation*}
\end{enumerate}
\end{taggedhyp}

\begin{rmk}[Local variant of hypotheses \ref{hyp:CE}]
In the sequel, we say that a Carathéodory vector-field $w: [0,T] \times \R^d \rightarrow \R^d$ which is \textnormal{independent from} $\mu(\cdot)$ satisfies hypotheses \ref{hyp:CE} if it is sub-linear and locally Lipschitz with respect to the space variable, with constants $m(\cdot),l_K(\cdot) \in L^1([0,T],\R_+)$.
\end{rmk}

In what follows, we recall structural results for continuity inclusions in Wasserstein spaces which were derived in \cite{ContInc}. From now on, we suppose that $V : [0,T] \times \Pcal_c(\R^d) \rightrightarrows C^0(\R^d,\R^d)$ is a set-valued map satisfying hypotheses \ref{hyp:DI}, and consider $\Pcal_c(\R^d)$ as a subset of the metric space $(\Pcal_p(\R^d),W_p)$.

\begin{thm}[Existence and estimates on solutions of continuity inclusions]
\label{thm:ExistenceWass}
For every $\mu^0 \in \Pcal_c(\R^d)$, there exists a curve of measures $\mu(\cdot) \in \AC([0,T],\Pcal_c(\R^d))$ solution of the continuity inclusion \eqref{eq:WassIncDef} such that $\mu(0) = \mu^0$. Moreover if $\mu^0 \in \Pcal(B(0,r))$ for some $r > 0$, then there exist $R_r > 0$ and $m_r(\cdot) \in L^1([0,T],\R_+)$ depending only on the magnitudes of $r,\Norm{m(\cdot)}_1$ such that
\begin{equation}
\label{eq:SuppAC_Est}
\supp(\mu(t)) \subset B(0,R_r) \qquad \text{and} \qquad W_p(\mu(t),\mu(s)) \leq \INTSeg{m_r(\tau)}{\tau}{s}{t},
\end{equation}
for every solution $\mu(\cdot)$ of \eqref{eq:WassIncDef} and all times $0 \leq s \leq t \leq T$.
\end{thm}

In the following corollary, we state elementary consequences of Theorem \ref{thm:ExistenceWass} in the case where $V(\cdot,\cdot)$ is single-valued. We present this result separately for the sake of clarity, as we shall frequently use this variant in the sequel. Its last statement follows e.g. from \cite[Chapter 8]{AGS} and \cite{Pedestrian}.

\begin{cor}[Non-local continuity equations]
\label{cor:NonLocalPDE}
In Theorem \ref{thm:ExistenceWass} suppose that $V : [0,T] \times \Pcal_c(\R^d) \rightrightarrows C^0(\R^d,\R^d)$ is single-valued, and let $\mu^0 \in \Pcal(B(0,r))$ for some $r > 0$. Then the non-local velocity-field $v : [0,T] \times \Pcal_c(\R^d) \times \R^d \rightarrow \R^d$ defined for all $(t,\mu,x) \in [0,T] \times \Pcal_c(\R^d) \times \R^d$ by
\begin{equation*}
v(t,\mu,x) := V(t,\mu)(x),
\end{equation*}
satisfies hypotheses \ref{hyp:CE}, and the non-local Cauchy problem \eqref{eq:NonLocalCE} admits a unique solution $\mu(\cdot) \in \AC([0,T],\Pcal_c(\R^d))$. Moreover, the curve $\mu(\cdot)$ satisfies the estimates of \eqref{eq:SuppAC_Est} and can be represented as
\begin{equation}
\label{eq:ExpressionMeasure}
\mu(t) = \Phi_{(0,t)}[\mu^0](\cdot)_{\#} \mu^0,
\end{equation}
for all times $t \in [0,T]$. Here, $(\Phi_{(s,t)}[\mu(s)](\cdot))_{s,t \in [0,T]}$ denotes the family of \textnormal{non-local flows} defined by
\begin{equation}
\label{eq:FlowExp1}
\Phi_{(s,t)}[\mu(s)](x) = x + \INTSeg{ v \Big( \sigma , \mu(\sigma) , \Phi_{(s,\sigma)}[\mu(s)](x) \Big)}{\sigma}{s}{t},
\end{equation}
for all times $s,t \in [0,T]$ and any $x \in \R^d$.
\end{cor}

We end this section by recalling three of the main results of \cite{ContInc}, which are generalisations to the setting of continuity inclusions in Wasserstein spaces of Filippov's estimates, the Relaxation theorem, and the one-to-one correspondence between control systems and continuity inclusions.

\begin{thm}[Filippov's theorem]
\label{thm:FilippovEstimate}
Let $\nu(\cdot) \in \AC([0,T],\Pcal(K_{\nu}))$ be a solution of \eqref{eq:CE} generated by a Carathéodory velocity-field $w : [0,T] \times \R^d \rightarrow \R^d$ with $K_{\nu} := B(0,R_{\nu})$ for some $R_{\nu} > 0$. Furthermore, suppose that the \textnormal{mismatch function} $\eta_{\nu} : [0,T] \rightarrow \R_+$ defined by
\begin{equation*}
\eta_{\nu}(t) := \dist_{C^0(K_{\nu},\R^d)} \Big( w_{\vert K_{\nu}}(t) ; V_{K_{\nu}}(t,\nu(t)) \Big),
\end{equation*}
is Lebesgue integrable over $[0,T]$. Then for any $\mu^0 \in \Pcal(B(0,r))$ with $r > 0$, there exists a solution $\mu(\cdot) \in \AC([0,T],\Pcal_c(\R^d))$ of the continuity inclusion \eqref{eq:WassIncDef} such that for all times $t \in [0,T]$, it holds
\begin{equation*}
W_p(\mu(t),\nu(t)) \leq C_p \left( W_p(\mu^0,\nu(0)) + \INTSeg{\eta_{\nu}(s)}{s}{0}{t} \right),
\end{equation*}
where $C_p > 0$ is a constant which depends only on the magnitudes of $p,r,R_{\nu},T$ and $\Norm{m(\cdot)}_1$.
\end{thm}

\begin{thm}[Relaxation theorem]
\label{thm:RelaxationWass}
Let $\mu^0 \in \Pcal_c(\R^d)$ and $\mu(\cdot) \in \AC([0,T],\Pcal_c(\R^d))$ be a solution of the \textnormal{relaxed continuity inclusion}
\begin{equation*}
\left\{
\begin{aligned}
& \partial_t \mu(t) \in - \Div \Big( \co V(t,\mu(t)) \mu(t) \Big), \\
& \mu(0) = \mu^0,
\end{aligned}
\right.
\end{equation*}
where $\co V(t,\mu(t))$ denotes the closed convex hull of $V(t,\mu(t)) \subset C^0(\R^d,\R^d)$ defined as in \eqref{eq:ConvbarDef}. Then for any $\delta > 0$, there exists a solution $\mu_{\delta}(\cdot) \in \AC([0,T],\Pcal_c(\R^d))$ of the continuity inclusion
\begin{equation*}
\left\{
\begin{aligned}
& \partial_t \mu_{\delta}(t) \in - \Div \Big( V(t,\mu_{\delta}(t)) \mu_{\delta}(t) \Big), \\
& \mu_{\delta}(0) = \mu^0,
\end{aligned}
\right.
\end{equation*}
such that
\begin{equation*}
W_p(\mu(t),\mu_{\delta}(t)) \leq \delta,
\end{equation*}
for all times $t \in [0,T]$.
\end{thm}

\begin{thm}[Correspondence between control systems and continuity inclusions]
\label{thm:ControlInc}
Let $(U,d_U)$ be a compact metric space and $v : [0,T] \times \Pcal_c(\R^d) \times U \times \R^d \rightarrow \R^d$ be a controlled non-local velocity-field such that $(t,\mu,x) \mapsto v(t,\mu,u,x) \in \R^d$ satisfies hypotheses \ref{hyp:CE} for every $u \in U$ with constants $m(\cdot),l_k(\cdot),L_K(\cdot) \in L^1([0,T],\R_+)$ independent of $u \in U$, and also $u \in U \mapsto v(t,\mu,u,x) \in \R^d$ is continuous for $\Lcal^1$-almost every $t \in [0,T]$ and any $(\mu,x) \in \Pcal_c(\R^d) \times \R^d$. Moreover, define the set-valued map $V : [0,T] \times \Pcal_c(\R^d) \rightrightarrows C^0(\R^d,\R^d)$ by
\begin{equation*}
V(t,\mu) := \Big\{ \vb \in C^0(\R^d,\R^d) ~\text{s.t.}~ \vb(\cdot) = v(t,\mu,u,\cdot) ~\text{for some $u \in U$} \Big\},
\end{equation*}
for all $(t,\mu) \in [0,T] \times \Pcal_c(\R^d)$.

Then, $V(\cdot,\cdot)$ satisfies hypotheses \ref{hyp:DI}, and a curve of measures $\mu(\cdot)$ is a solution of the continuity inclusion \eqref{eq:WassIncDef} \textnormal{if and only if} it solves the controlled non-local continuity equation
\begin{equation*}
\partial_t \mu(t) + \Div \big( v(t,\mu(t),u(t)) \mu(t) \big) = 0,
\end{equation*}
for some $\Lcal^1$-measurable selection $t \in [0,T] \mapsto u(t) \in U$.
\end{thm}


\section{Differential calculus and linearised dynamics in $\Pcal_c(\R^d)$}
\label{section:NonSmoothWass}

In this section, we investigate novel differential calculus tools for functionals and dynamical systems defined over $\Pcal_c(\R^d)$ seen as a subset of $(\Pcal_2(\R^d),W_2)$. In Section \ref{subsection:LocalisedSubdiff}, we introduce a localised notion of $\Pcal_2(\R^d)$-differentiability for compactly supported measures, and prove a new chain rule formula along arbitrary transport plans. In Section \ref{subsection:Linearisation}, we derive the general expression for the linearised Cauchy problem associated to a non-local continuity equation, and focus on the particular case in which the velocity perturbations are tangent to the set of admissible velocities of a continuity inclusion.


\subsection{Localised theory of Wasserstein calculus}
\label{subsection:LocalisedSubdiff}

In this section, we study the properties of a new notion of \textit{localised} subdifferential defined in the spirit of \cite{AGS,Gangbo2019} for extended real-valued functionals $\phi : \Pcal_2(\R^d) \rightarrow \R \cup \{ \pm \infty \}$ such that $\Pcal_c(\R^d) \subset D(\phi) := \{ \mu \in \Pcal_2(\R^d) ~\text{s.t.}~ \phi(\mu) \neq \pm \infty \}$. In the sequel, we will use the condensed notation $\phi : \Pcal_c(\R^d) \rightarrow \R$ to refer to any such functional.

\begin{Def}[Localised sub and superdifferential]
\label{def:Subdiff}
Let $\phi : \Pcal_c(\R^d) \rightarrow \R$ and $\mu \in \Pcal_c(\R^d)$. We say that a map $\xi \in L^2(\R^d,\R^d;\mu)$ belongs to the \textnormal{localised subdifferential} $\partial^-_{\loc} \phi(\mu)$ of $\phi(\cdot)$ at $\mu$ provided that
\begin{equation}
\label{eq:Subdiff}
\phi(\nu) - \phi(\mu) \geq \inf_{\gamma \in \Gamma_o(\mu,\nu)} \INTDom{\langle \xi(x) , y-x \rangle}{\R^{2d}}{\gamma(x,y)} + o_R(W_2(\mu,\nu)),
\end{equation}
for every $R > 0$ and any $\nu \in \Pcal(B_{\mu}(R))$, with $B_{\mu}(R)$ being defined as in \eqref{eq:Fattening}. Furthermore, we say that a map $\xi \in L^2(\R^d,\R^d;\mu)$ belongs to the \textnormal{strong localised subdifferential} $\partial^-_{S,\loc} \phi(\mu)$ if it satisfies 
\begin{equation}
\label{eq:StrongSubdiff}
\phi(\nu) - \phi(\mu) \geq \INTDom{\langle \xi(x) , y-x \rangle}{\R^{2d}}{\Bmu(x,y)} + o_R(W_{2,\Bmu}(\mu,\nu)),
\end{equation}
for every $R> 0$, any $\nu \in \Pcal(B_{\mu}(R))$ and each $\Bmu \in \Gamma(\mu,\nu)$, where we introduced the notation
\begin{equation}
\label{eq:WeightedWass}
W_{2,\Bmu}(\mu,\nu) = \left( \INTDom{|x-y|^2}{\R^{2d}}{\Bmu(x,y)} \right)^{1/2}.
\end{equation}
Analogously, we say that $\xi \in L^2(\R^d,\R^d;\mu)$ belongs to the \textnormal{localised superdifferential} $\partial^+_{\loc} \phi(\mu)$ of $\phi(\cdot)$ at $\mu$ if $(-\xi) \in \partial^-_{\loc} (-\phi)(\mu)$, with a similar definition for strong localised superdifferentials.
\end{Def}

\begin{rmk}[On the difference between \eqref{eq:Subdiff} and \eqref{eq:StrongSubdiff}]
In line with the theory of subdifferential calculus on manifolds, localised subdifferentials are defined via their action on optimal displacement directions. As illustrated first in \cite{McCann1997} (see also \cite[Chapter 7]{AGS}), such optimal directions are given by displacement interpolations of optimal plans, which leads to the geometric definition \eqref{eq:Subdiff}.

On the other hand, defining localised subdifferentials via their action on any transport direction as in \eqref{eq:StrongSubdiff} produces a class of subdgradients that is much smaller, and which does not enjoy several of the desirable features of classical subdifferentials (such as being weakly closed, see e.g. \cite[Remark 10.3.2]{AGS}). Yet, strong subdifferentials are extremely useful in practice as they allow to differentiate functionals along directions that are not necessarily optimal. This property is of crucial interest in control theory, as the corresponding classes of admissible variations usually belong to this category.
\end{rmk}

By a direct adaptation of the proofs of \cite[Section 3]{Gangbo2019} with finite radii $R > 0$, it can be shown that $\partial_{\loc}^- \phi(\mu) \cap \partial^+_{\loc} \phi(\mu)$ contains at most one element, which also belongs to $\Tan_{\mu} \Pcal_2(\R^d)$. This motivates the following definition of \textit{local differentiability} for functionals defined over $\Pcal_c(\R^d)$.

\begin{Def}[Locally differentiable functional]
\label{def:LocalDiff}
A functional $\phi : \Pcal_c(\R^d) \rightarrow \R$ is \textnormal{locally differentiable at $\mu \in \Pcal_c(\R^d)$} if there exists a map $\nabla \phi(\mu) \in \Tan_{\mu}\Pcal_2(\R^d)$ -- called the \textnormal{Wasserstein gradient} of $\phi(\cdot)$ at $\mu$ --, such that $\partial^-_{\loc} \phi(\mu) \cap \partial^+_{\loc} \phi(\mu) = \{ \nabla \phi(\mu) \}$. Similarly, we shall say that $\phi(\cdot)$ is locally differentiable over $\Pcal_c(\R^d)$ if it is locally differentiable at every $\mu \in \Pcal_c(\R^d)$.
\end{Def}
 
\begin{rmk}[Uniformity of the gradient]
Observe that in Definition \ref{def:LocalDiff} above, we impose on the gradient $\nabla \phi(\mu) \in \partial^-_{\loc} \phi(\mu) \cap \partial^+_{\loc} \phi(\mu)$ to be independent from $R > 0$. This is due to the fact that $\nabla \phi(\mu)$ is an element of $\Tan_{\mu} \Pcal_2(\R^d)$ -- thus defined $\mu$-almost everywhere over $\R^d$ --, so that it should not depend on the fattening parameter of the support of $\nu \in \Pcal(B_{\mu}(R))$. The latter is solely used to control the small-o of the $W_2$-distance in \eqref{eq:Subdiff}, as illustrated in Proposition \ref{prop:Example} below.
\end{rmk}

We provide examples of locally differentiable functionals in Section \ref{appendix:Examples}. In the next proposition, we prove that localised subdifferentials which belong to the analytical tangent space $\Tan_{\mu} \Pcal_2(\R^d)$ are in fact strong localised subdifferentials. The proof of this result is inspired from that of \cite[Chapter 1, Proposition 4.2]{Dafermos2006}, where a similar property is established for classical subdifferentials and measures $\mu \in \Pcal_2(\R^d)$ which are \textit{absolutely continuous} with respect to the standard Lebesgue measure $\Lcal^d$.

\begin{prop}[Tangent localised subdifferentials are strong]
\label{prop:StrongSubdiff}
Let $\phi : \Pcal_c(\R^d) \rightarrow \R$ and $\mu \in \Pcal_c(\R^d)$ be such that $\partial^-_{\loc} \phi(\mu) \neq \emptyset$. Then
\begin{equation*}
\partial_{\loc}^- \phi(\mu) \cap \Tan_{\mu} \Pcal_2(\R^d) \subset \partial_{S,\loc}^- \phi(\mu),
\end{equation*}
i.e. any localised subdifferential which belongs to the analytical tangent space $\Tan_{\mu} \Pcal_2(\R^d)$ is also a strong localised subdifferential.
\end{prop}

\begin{proof}
By contradiction, suppose that there exists $\xi \in \partial_{\loc}^- \phi(\mu) \cap \Tan_{\mu} \Pcal_2(\R^d)$ which is not a strong localised subdifferential. Then for some $R> 0$ and $\delta >0$, we can find a sequence of measures $(\mu_n) \subset \Pcal(B_{\mu}(R))$ which converges towards $\mu$ in the $W_2$-metric along with a sequence of plans $(\Bmu_n) \subset \Gamma(\mu,\mu_n)$ such that for any $n \geq 1$ large enough, it holds
\begin{equation}
\label{eq:StrongSubdiffIneq}
\phi(\mu_n) - \phi(\mu) - \INTDom{\langle \xi(x) , y-x \rangle}{\R^{2d}}{\Bmu_n(x,y)} \leq - \delta \epsilon_n,
\end{equation}
where $\epsilon_n := W_{2,\Bmu_n}(\mu,\mu_n)$ is given by \eqref{eq:WeightedWass} and satisfies $\epsilon_n \rightarrow 0^+$ as $n \rightarrow +\infty$. Choose $\gamma_n \in \Gamma_o(\mu,\mu_n)$ for any $n \geq 1$, and observe that since $\xi \in \partial_{\loc}^- \phi(\mu)$, it holds for $n \geq 1$ large enough 
\begin{equation}
\label{eq:SubdiffIneq}
\phi(\mu_n) - \phi(\mu) - \INTDom{\langle \xi(x) , y-x \rangle}{\R^{2d}}{\gamma_n(x,y)} \geq - \tfrac{\delta}{2} \epsilon_n,
\end{equation}
where we used the fact that $W_2(\mu,\mu_n) \leq W_{2,\Bmu_n}(\mu,\mu_n)$. By merging \eqref{eq:StrongSubdiffIneq} and \eqref{eq:SubdiffIneq}, we obtain
\begin{equation}
\label{eq:Ineq1}
\INTDom{\big\langle \xi(x) , \tfrac{y-x}{\epsilon_n} \big\rangle}{\R^{2d}}{(\gamma_n - \Bmu_n)(x,y)} \leq - \tfrac{\delta}{2}.
\end{equation}
We now introduce the sequences of rescaled plans $(\hat{\gamma}_n)$ and $(\hat{\Bmu}_n)$, defined for all $n \geq 1$ by
\begin{equation*}
\hat{\gamma}_n := \big( \pi^1 , \tfrac{\pi^2-\pi^1}{\epsilon_n} \big)_{\#} \gamma_n \qquad \text{and} \qquad \hat{\Bmu}_n := \big( \pi^1 , \tfrac{\pi^2-\pi^1}{\epsilon_n} \big)_{\#} \Bmu_n,
\end{equation*}
so that \eqref{eq:Ineq1} can be rewritten as
\begin{equation}
\label{eq:Ineq1Bis}
\INTDom{\langle \xi(x) , r \rangle}{\R^{2d}}{(\hat{\gamma}_n - \hat{\Bmu}_n)(x,r)} \leq - \tfrac{\delta}{2}.
\end{equation}
We claim that both sequences $(\hat{\gamma}_n)$ and $(\hat{\Bmu}_n)$ have uniformly bounded second order momentum. Indeed
\begin{equation*}
\sup_{n \geq 1} \M_2^2(\hat{\gamma}_n) = \sup_{n \geq 1} \INTDom{\Big( |r|^2 + |x|^2 \Big)}{\R^{2d}}{\hat{\gamma}_n(x,r)} = \sup_{n \geq 1} \INTDom{\Big( \big| \tfrac{x-y}{\epsilon_n} \big|^2 + |x|^2 \Big)}{\R^{2d}}{\gamma_n(x,y)}  \leq 1 + \M_2^2(\mu),
\end{equation*}
and the same estimate also holds for $(\hat{\Bmu}_n)$. Whence, $(\hat{\gamma}_n)$ and $(\hat{\Bmu}_n)$ are tight (see e.g. \cite[Remark 5.1.5]{AGS}), and thus narrowly compact by Prokhorov's theorem (see e.g. \cite[Theorem 5.1.3]{AGS}). In the sequel, we denote by $\hat{\gamma}$ and $\hat{\Bmu}$ two of their cluster points that can be reached along the same subsequence. 

Our goal now is to pass to the limit as $n \rightarrow +\infty$ in \eqref{eq:Ineq1Bis}. Let $(\xi_k) \subset C^{\infty}_c(\R^d,\R^d)$ be a sequence such that $\xi_k \rightarrow \xi$ in $L^2(\R^d,\R^d;\mu)$ as $k \rightarrow +\infty$, and observe that
\begin{equation}
\label{eq:Estimate_kn1}
\begin{aligned}
\INTDom{\langle \xi_k(x) , r \rangle}{\R^{2d}}{\hat{\gamma}_n(x,r)} & \leq \INTDom{\langle \xi(x) , r \rangle}{\R^{2d}}{\hat{\gamma}_n(x,r)} + \bigg( \INTDom{|\xi(x) - \xi_k(x)|^2}{\R^d}{\mu(x)} \bigg)^{\hspace{-0.1cm} 1/2} \hspace{-0.05cm} \bigg( \INTDom{|r|^2}{\R^{2d}}{\hat{\gamma}_n(x,r)} \bigg)^{\hspace{-0.1cm} 1/2} \\
& \leq \INTDom{\langle \xi(x) , r \rangle}{\R^{2d}}{\hat{\gamma}_n(x,r)} + \NormL{\xi - \xi_k}{2}{\mu}, 
\end{aligned}
\end{equation}
where we used Cauchy-Schwarz's inequality and the definition of $(\hat{\gamma}_n)$. Similarly, one also has
\begin{equation}
\label{eq:Estimate_kn2}
\INTDom{\langle \xi_k(x) , r \rangle}{\R^{2d}}{\hat{\Bmu}_n(x,r)} \geq \INTDom{\langle \xi(x) , r \rangle}{\R^{2d}}{\hat{\Bmu}_n(x,r)} - \NormL{\xi - \xi_k}{2}{\mu}
\end{equation}
for all $n,k \geq 1$. Whence by plugging \eqref{eq:Estimate_kn1} and \eqref{eq:Estimate_kn2} into \eqref{eq:Ineq1Bis}, one deduces for all $n,k \geq 1$ large enough
\begin{equation}
\label{eq:Estimate_kn3}
\INTDom{\langle \xi_k(x) , r \rangle}{\R^{2d}}{(\hat{\gamma}_n - \hat{\Bmu}_n)(x,r)} \leq - \tfrac{\delta}{4},
\end{equation}
We now claim that given $k \geq 1$, the map $(x,r) \in \R^{2d} \mapsto |\langle \xi_k(x) ,r \rangle| \in \R_+$ is uniformly integrable with respect to both $(\hat{\gamma}_n)$ and $(\hat{\Bmu}_n)$. Indeed setting $C_k := \NormC{\xi_k(\cdot)}{0}{\R^d,\R^d}$, it holds 
\begin{equation}
\label{eq:UnifInteg1}
\begin{aligned}
\INTDom{ |\langle \xi_k(x) , r \rangle| }{\{(x,r) ~\text{s.t.}~ |\langle \xi_k(x) ,r \rangle| \geq i \}}{\hat{\gamma}_n(x,r)} & \leq C_k \INTDom{|r| }{\{(x,r) ~\text{s.t.}~ |r| \geq i/C_k \}}{\hat{\gamma}_n(x,r)}, 
\end{aligned}
\end{equation}
for every $i \geq 1$, while on the other hand it can be checked by definition of $(\hat{\gamma}_n)$ that 
\begin{equation}
\label{eq:UnifInteg2}
\frac{i}{C_k} \INTDom{|r| }{\{(x,r) ~\text{s.t.}~ |r| \geq i/C_k \}}{\hat{\gamma}_n(x,r)} \leq \INTDom{|r|^2}{\R^{2d}}{\hat{\gamma}_n(x,r)}  \leq 1, 
\end{equation}
for any $n,k,i \geq 1$. Merging \eqref{eq:UnifInteg1} and \eqref{eq:UnifInteg2}, we therefore obtain
\begin{equation*}
\INTDom{ |\langle \xi_k(x) , r \rangle| }{\{(x,r) ~\text{s.t.}~ |\langle \xi_k(x) ,r \rangle| \geq i \}}{\hat{\gamma}_n(x,r)} \leq \tfrac{C_k^2}{i} ~\underset{i \rightarrow +\infty}{\longrightarrow}~ 0,
\end{equation*}
so that the $(x,r) \in \R^{2d} \mapsto |\langle \xi_k(x),r\rangle| \in \R_+$ is uniformly integrable with respect to $(\hat{\gamma}_n)$ for all $k \geq 1$. By repeating the same arguments for $(\hat{\Bmu}_n)$ and recalling that the latter map is continuous, we recover
\begin{equation}
\label{eq:UnifInteg3}
\INTDom{\langle \xi_k(x) , r \rangle}{\R^{2d}}{(\hat{\gamma}_n - \hat{\Bmu}_n)(x,r)} ~\underset{n \rightarrow +\infty}{\longrightarrow}~ \INTDom{\langle \xi_k(x) , r \rangle}{\R^{2d}}{(\hat{\gamma} - \hat{\Bmu})(x,r)}, 
\end{equation}
for all $k \geq 1$ by applying the convergence results of \cite[Lemma 5.1.7]{AGS}. Upon combining \eqref{eq:Estimate_kn3} and \eqref{eq:UnifInteg3} while letting $k \rightarrow +\infty$, we can thus conclude
\begin{equation}
\label{eq:Ineq2}
\INTDom{\langle \xi(x) , r \rangle}{\R^{2d}}{(\hat{\gamma} - \hat{\Bmu})(x,r)} \leq - \tfrac{\delta}{4}, 
\end{equation}
where we used the fact that $\pi^1_{\#} \hat{\gamma} = \pi^1_{\#} \hat{\Bmu} = \mu$ as a consequence of classical stability results under narrow convergence for transport plans (see e.g. \cite[Chapter 5]{AGS}). 

Let us now choose an arbitrary test function $\zeta \in C^{\infty}_c(\R^d)$. First, notice that there exists a constant $C_{\zeta} > 0$ such that the global estimates
\begin{equation}
\label{eq:EstimatesTestFunction}
\left\{
\begin{aligned}
& \zeta(y) - \zeta(x) \leq \langle \nabla \zeta(x) , y-x \rangle  + C_{\zeta} |x-y|^2, \\
& \zeta(x) - \zeta(y) \leq \langle \nabla \zeta(x) , x-y \rangle  + C_{\zeta} |x-y|^2,
\end{aligned}
\right.
\end{equation}
hold for any $x,y \in \R^d$. By using the first estimates of \eqref{eq:EstimatesTestFunction}, we further obtain for any $n \geq 1$
\begin{equation}
\label{eq:TanEst1}
\INTDom{\zeta(y)}{\R^{2d}}{\gamma_n(x,y)} \leq  \INTDom{\Big( \zeta(x) + \langle \nabla \zeta(x) , y-x \rangle \Big)}{\R^{2d}}{\gamma_n(x,y)} + C_{\zeta} W_2^2(\mu,\mu_n),
\end{equation}
where we used the fact that $\gamma_n \in \Gamma_o(\mu,\mu_n)$, while the second estimate of \eqref{eq:EstimatesTestFunction} further yields
\begin{equation}
\label{eq:TanEst2}
- \INTDom{\zeta(y)}{\R^{2d}}{\Bmu_n(x,y)} \leq -\INTDom{\Big( \zeta(x) + \langle \nabla \zeta(x) , y-x \rangle \Big)}{\R^{2d}}{\Bmu_n(x,y)} + C_{\zeta} W_{2,\Bmu_n}^2(\mu,\mu_n),
\end{equation}
by the definition \eqref{eq:WeightedWass} of $W_{2,\Bmu_n}(\mu,\mu_n)$. In addition, observe that since $\gamma_n,\Bmu_n \in \Gamma(\mu,\mu_n)$, one also has
\begin{equation}
\label{eq:Zeta_ZeroEq}
\INTDom{\zeta(x)}{\R^{2d}}{(\gamma_n - \Bmu_n)(x,y)} = 0 \qquad \text{and} \qquad \INTDom{\zeta(y)}{\R^{2d}}{(\gamma_n - \Bmu_n)(x,y)} = 0,
\end{equation}
for all $n \geq 1$. Therefore, combining the estimates of \eqref{eq:TanEst1} and \eqref{eq:TanEst2} while using \eqref{eq:Zeta_ZeroEq} and recalling that $W_2(\mu,\mu_n) \leq W_{2,\Bmu_n}(\mu,\mu_n) = \epsilon_n$, we obtain
\begin{equation*}
0 \leq \INTDom{\langle \nabla \zeta(x) , y-x \rangle}{\R^{2d}}{(\gamma_n - \Bmu_n)(x,y)} + 2C_{\zeta} \epsilon_n^2.
\end{equation*}
By dividing both sides of the previous inequality by $\epsilon_n >0$, considering again the sequences of rescaled plans $(\hat{\gamma}_n)$ and $(\hat{\Bmu}_n)$ and letting $n \rightarrow +\infty$ along an adequate subsequence, we finally recover
\begin{equation}
\label{eq:Ineq3}
0 \leq \INTDom{\langle \nabla \zeta(x) , r \rangle}{\R^{2d}}{(\hat{\gamma} - \hat{\Bmu})(x,r)},
\end{equation}
for any $\zeta \in C^{\infty}_c(\R^d)$. Since $\xi \in \Tan_{\mu} \Pcal_2(\R^d)$, there exists a sequence $(\zeta_n) \subset C^{\infty}_c(\R^d)$ such that $\nabla \zeta_n \rightarrow \xi$ in $L^2(\R^d,\R^d;\mu)$ as $n \rightarrow +\infty$, which leads to a contradiction between \eqref{eq:Ineq2} and \eqref{eq:Ineq3}.
\end{proof}

By combining Definition \ref{def:LocalDiff} and Proposition \ref{prop:StrongSubdiff}, we obtain a general chain rule along arbitrary transport plans for locally differentiable functionals.

\begin{prop}[Chain rule along arbitrary transport plans]
\label{prop:Chainrule}
Let $\phi : \Pcal_c(\R^d) \rightarrow \R$, $\mu \in \Pcal_c(\R^d)$ and suppose that $\phi(\cdot)$ is locally differentiable at $\mu$. Then for every $R > 0$ and $\nu \in \Pcal(B_{\mu}(R))$, it holds
\begin{equation}
\label{eq:Chainrule}
\phi(\nu) = \phi(\mu) + \INTDom{\langle \nabla\phi(\mu) (x) , y-x \rangle}{\R^{2d}}{\Bmu(x,y)} + o_R(W_{2,\Bmu}(\mu,\nu)),
\end{equation}
for any $\Bmu \in \Gamma(\mu,\nu)$. Conversely, if there exists a map $\nabla \phi(\mu) \in \Tan_{\mu} \Pcal_2(\R^d)$ such that for every $R >0$ and $\nu \in \Pcal(B_{\mu}(R))$ the first-order expansion \eqref{eq:Chainrule} holds along any optimal plan $\Bmu \in \Gamma_o(\mu,\nu)$, then $\phi(\cdot)$ is locally differentiable at $\mu$ in the sense of Definition \ref{def:LocalDiff}.
\end{prop}

\begin{proof}
By Definition \ref{def:LocalDiff}, we have that $\nabla \phi(\mu) \in \partial^-_{\loc} \phi(\mu) \cap \partial^+_{\loc} \phi(\mu) \cap \Tan_{\mu} \Pcal_2(\R^d)$. Whence by Proposition \ref{prop:StrongSubdiff}, it further holds that $\nabla \phi(\mu) \in \partial^-_{S,\loc} \phi(\mu) \cap \partial^+_{S,\loc} \phi(\mu)$, so that \eqref{eq:Chainrule} is satisfied for every $\nu \in \Pcal(B_{\mu}(R))$ and any $\Bmu \in \Gamma(\mu,\nu)$. Conversely, let $\nu \in \Pcal(B_{\mu}(R))$ for some $R > 0$, and observe that for every optimal transport plan $\Bmu \in \Gamma_o(\mu,\nu)$, it holds
\begin{equation*}
\INTDom{\langle \nabla \phi(\mu)(x), y-x \rangle}{\R^{2d}}{\Bmu(x,y)} \geq \inf_{\gamma \in \Gamma_o(\mu,\nu)} \INTDom{\langle \nabla \phi(\mu), y-x \rangle}{\R^{2d}}{\gamma(x,y)},
\end{equation*}
so that \eqref{eq:Chainrule} directly yields that $\nabla \phi(\mu) \in \partial_{\loc}^-\phi(\mu)$ since $W_{2,\Bmu}(\mu,\nu) = W_2(\mu,\nu)$ when $\Bmu \in \Gamma_o(\mu,\nu)$ and both $R > 0$ and $\nu \in \Pcal(B_{\mu}(R))$ are arbitrary. Similarly, one can show that $\nabla \phi(\mu) \in \partial^+_{\loc} \phi(\mu)$. Hence, $\nabla \phi(\mu) \in \partial^-_{\loc} \phi(\mu) \cap \partial^+_{\loc} \phi(\mu) \cap \Tan_{\mu} \Pcal_2(\R^d)$, and the functional $\phi(\cdot)$ is locally differentiable at $\mu$ in the sense of Definition \ref{def:LocalDiff}.
\end{proof}

\begin{cor}[Chain rule along perturbations in $L^{\infty}(\R^d,\R^d;\mu)$]
\label{cor:ChainruleBis}
Let $\phi : \Pcal_c(\R^d) \rightarrow \R$, $\mu \in \Pcal_c(\R^d)$ and suppose that $\phi(\cdot)$ is locally differentiable at $\mu$. Then for every $\F \in L^{\infty}(\R^d,\R^d;\mu)$, it holds
\begin{equation*}
\phi \big( (\Id + \epsilon \F)_{\#}\mu \big) = \phi(\mu) + \epsilon \INTDom{\langle \nabla \phi(\mu)(x) , \F(x)\rangle}{\R^d}{\mu(x)} + o_{R_{\F}}(\epsilon),
\end{equation*}
for all $\epsilon \in [-1,1]$, where $R_{\F} := \NormL{\F}{\infty}{\mu}$.
\end{cor}

\begin{proof}
Observe first that for every $\F \in L^{\infty}(\R^d,\R^d;\mu)$, the measures $(\Id + \epsilon \F)_{\#} \mu$ have compact support in $B_{R_{\F}}(\mu)$ for all $\epsilon \in [-1,1]$. By applying Proposition \ref{prop:Chainrule} with the transport plan $\Bmu_{\epsilon} := (\Id,\Id+\epsilon \F)_{\#} \mu$, we therefore obtain
\begin{equation*}
\begin{aligned}
\phi \big( (\Id + \epsilon \F)_{\#}\mu \big) & = \phi(\mu) + \INTDom{\langle \nabla \phi(\mu)(x), y-x \rangle}{\R^{2d}}{\Bmu_{\epsilon}(x,y)} + o_{R_{\F}} \big( W_{2,\Bmu_{\epsilon}}(\mu,(\Id+\epsilon \F)_{\#}\mu) \big) \\
& = \phi(\mu) + \epsilon \INTDom{\langle \nabla \phi(\mu)(x), \F(x) \rangle}{\R^d}{\mu(x)} + o_{R_{\F}}(\epsilon),
\end{aligned}
\end{equation*}
where we used the fact that $W_{2,\Bmu_{\epsilon}}(\mu,(\Id+\epsilon \F)_{\#}\mu) = \epsilon \NormL{\F}{2}{\mu}$ as a direct consequence of \eqref{eq:WeightedWass}.
\end{proof}

We end this section by adapting the notion of local differentiability to vector-valued maps.

\begin{Def}[Local differentiability of vector-valued maps]
\label{def:VectorLocalDiff}
For $m \geq 1$, a map $\phi : \Pcal_c(\R^d) \rightarrow \R^m$ is said to be \textnormal{locally differentiable} at $\mu \in \Pcal_c(\R^d)$ if its components $\{\phi_i(\cdot)\}_{1 \leq i \leq m}$ are locally differentiable.
\end{Def}

\begin{rmk}[Chainrule for vector-valued maps]
\label{rmk:VectVal}
When $m = d$ and $\phi : \Pcal_c(\R^d) \rightarrow \R^d$ is locally differentiable at $\mu \in \Pcal_c(\R^d)$, we can apply Proposition \ref{prop:Chainrule} componentwisely to obtain
\begin{equation}
\label{eq:TaylorMeasure_Vectorvalued}
\phi(\nu) = \phi(\mu) + \INTDom{\D_{\mu} \phi(\mu)(x)(y-x)}{\R^{2d}}{\Bmu(x,y)} + o_{R}(W_{2,\Bmu}(\mu,\nu)),
\end{equation}
for every $\nu \in \Pcal(B_{\mu}(R))$ with $R > 0$, where $\D_{\mu} \phi(\mu)(\cdot) := (\nabla \phi_i(\mu)(\cdot))_{1 \leq i \leq d}$ is the matrix-valued map which rows are the Wasserstein gradients of the components $\{ \phi_i(\cdot) \}_{1 \leq i \leq d}$. Similarly, by applying componentwisely the statements of Corollary \ref{cor:ChainruleBis}, it also holds
\begin{equation}
\label{eq:TaylorMeasure_VectorvaluedBis}
\phi((\Id + \epsilon \F)_{\#} \mu) = \phi(\mu) + \epsilon \INTDom{\D_{\mu} \phi(\mu)(x)\F(x)}{\R^d}{\mu(x)} + o_{R_{\F}}(\epsilon),
\end{equation}
for all $\epsilon \in [-1,1]$ and any $\F \in L^{\infty}(\R^d,\R^d;\mu)$. 
\end{rmk}


\subsection{Linearisations of non-local continuity equations}
\label{subsection:Linearisation}

In this section, we derive the expression of the canonical linearised Cauchy problem associated to a non-local continuity equation with smooth driving velocity-field. In what follows we fix $p \in [1,+\infty)$ and for any compact set $K \subset \R^d$, we shall consider $\Pcal(K)$ as a subset of the metric space $(\Pcal_p(\R^d),W_p)$. We henceforth focus our attention on the Cauchy problem
\begin{equation}
\label{eq:ContinuityEquation_Nonlocal}
\left\{
\begin{aligned}
& \partial_t \mu(t) + \Div \big( v(t,\mu(t)) \mu(t) \big) = 0, \\
& \mu(0) = \mu^0 \in \Pcal_c(\R^d),
\end{aligned}
\right.
\end{equation}
driven by a non-local velocity-field $v : [0,T] \times \Pcal_c(\R^d) \times \R^d \mapsto \R^d$ satisfying the following assumptions.

\begin{taggedhyp}{\textbn{(H)}}
\label{hyp:H}
For every $R > 0$, assume that the following holds with $K := B(0,R)$.
\begin{enumerate}
\item[$(i)$] The non-local velocity field $v : [0,T] \times \Pcal_c(\R^d) \times \R^d \rightarrow \R^d$ satisfies hypotheses \ref{hyp:CE}.
\item[$(ii)$] For $\Lcal^1$-almost every $t \in [0,T]$ and any $\mu \in \Pcal_c(\R^d)$, the map $x \in \R^d \mapsto v(t,\mu,x) \in \R^d$ is Fréchet differentiable, and the application $(\mu,x) \in \Pcal(K) \times K \mapsto \D_x v(t,\mu,x) \in \R^{d \times d}$ is continuous.
\item[$(iii)$] For $\Lcal^1$-almost every $t \in [0,T]$ and any $x \in \R^d$, the map $\mu \in \Pcal_c(\R^d) \mapsto v(t,\mu,x) \in \R^d$ is locally differentiable in the sense of Definition \ref{def:VectorLocalDiff}, and the application $(\mu,x,y) \in \Pcal(K) \times K \times K \mapsto \D_{\mu} v(t,\mu,x)(y) \in \R^{d  \times d}$ is continuous.
\end{enumerate}
\end{taggedhyp}

In the following proposition, we derive a general first-order linearisation formula for solutions of \eqref{eq:ContinuityEquation_Nonlocal}. Similar techniques were applied in \cite{PMPWassConst,PMPWass} in the particular case of perturbations induced by \textit{needle variations} (see e.g. \cite{Pontryagin}).

\begin{prop}[Linearisation of non-local continuity equations]
\label{prop:Linearisation_Cauchy}
Let $\mu^0 \in \Pcal(B(0,r))$ for some $r > 0$, $v : [0,T] \times \Pcal_c(\R^d) \times \R^d \rightarrow \R^d$ be a non-local velocity-field satisfying hypotheses \ref{hyp:H}, and $\mu(\cdot) \in \AC([0,T],\Pcal(K))$ be the corresponding solution of \eqref{eq:ContinuityEquation_Nonlocal} where $K := B(0,R_r)$ is given by \eqref{eq:SuppAC_Est}. Let $\F^0 \in C^0(\R^d,\R^d)$, $w : [0,T] \times \R^d \rightarrow \R^d$ be a Carathéodory vector-field satisfying \ref{hyp:CE} and consider the perturbed Cauchy problems
\begin{equation}
\label{eq:Perturbed_NonLocal}
\left\{
\begin{aligned}
& \partial_t \mu_{\epsilon}(t) + \Div \Big( \big( v(t,\mu_{\epsilon}(t)) + \epsilon w(t) \big)\mu_{\epsilon}(t) \Big) = 0, \\
& \mu_{\epsilon}(0) = (\Id + \epsilon \F^0)_{\#} \mu^0.
\end{aligned}
\right.
\end{equation}
for $\epsilon \in [-1,1]$. 

Then, there exists a family of maps $(\G_{\epsilon}(\cdot,\cdot)) \subset C^0([0,T] \times K,\R^d)$ such that for every $\epsilon \in [-1,1]$, the solution $\mu_{\epsilon}(\cdot) \in \AC([0,T],\Pcal_c(\R^d))$ of \eqref{eq:Perturbed_NonLocal} can be expressed explicitly as
\begin{equation}
\label{eq:PerturbedMeasure_Expr}
\mu_{\epsilon}(t) = \G_{\epsilon}(t,\cdot)_{\#} \mu(t), 
\end{equation}
for all times $t \in [0,T]$. Moreover, the application $\epsilon \in [-1,1] \mapsto \G_{\epsilon}(t,\cdot) \in C^0(K,\R^d)$ admits a Taylor expansion at $\epsilon = 0$, of the form
\begin{equation}
\label{eq:TaylorMeasure}
\G_{\epsilon}(t,\cdot) = \Id + \epsilon \F \Big( t, \Phi_{(t,0)}[\mu(t)](\cdot) \Big) +  o_t(\epsilon),
\end{equation}
where $\sup_{t \in [0,T]} \NormC{o_t(\epsilon)}{0}{K,\R^d} = o(\epsilon)$ and $\F \in C^0([0,T] \times K,\R^d)$ is the unique solution of the \textnormal{linearised non-local Cauchy problem}
\begin{equation}
\label{eq:Linearised_CauchyNonLocal}
\left\{
\begin{aligned}
& \partial_t \F(t,x) = \D_x v \Big( t,\mu(t), \Phi_{(0,t)}[\mu^0](x) \Big) \F(t,x) + w \Big( t , \Phi_{(0,t)}[\mu^0](x) \Big) \\
& \hspace{2cm} + \INTDom{\D_{\mu} v \Big( t,\mu(t), \Phi_{(0,t)}[\mu^0](x) \Big) \big( \Phi_{(0,t)}[\mu^0](y) \big) \F(t,y)}{\R^d}{\mu^0(y)}, \\
& \F(0,x) = \F^0(x).
\end{aligned}
\right.
\end{equation}
\end{prop}

\begin{proof}
First, remark that for any $\epsilon \in [-1,1]$, the non-local velocity-field $v_{\epsilon} : [0,T] \times \Pcal_c(\R^d) \times \R^d \rightarrow \R^d$ defined for any $(t,\mu,x) \in [0,T] \times \Pcal_c(\R^d) \times \R^d$ by
\begin{equation}
\label{eq:v_epsilon}
v_{\epsilon}(t,\mu,x) := v(t,\mu,x) + \epsilon w(t,x),
\end{equation}
satisfies hypotheses \ref{hyp:CE}. Whence, denoting $\mu^0_{\epsilon} := (\Id + \epsilon \F^0)_{\#} \mu^0$, there exists by Corollary \ref{cor:NonLocalPDE} a unique curve of measures $\mu_{\epsilon}(\cdot)$ solution of \eqref{eq:Perturbed_NonLocal} which can be represented as
\begin{equation}
\label{eq:FlowProof0}
\begin{aligned}
\mu_{\epsilon}(t) & = \Phi^{\epsilon}_{(0,t)}[\mu_{\epsilon}^0](\cdot)_{\#}\mu_{\epsilon}^0,
\end{aligned}
\end{equation}
where $(\Phi^{\epsilon}_{(0,t)}[\mu^0_{\epsilon}](\cdot))_{t \in [0,T]}$ stands for the non-local flow generated by $v_{\epsilon} : [0,T] \times \Pcal_c(\R^d) \times \R^d \rightarrow \R^d$ starting from $\mu^0_{\epsilon}$, defined as in \eqref{eq:FlowExp1}. For the sake of readability, we shall use the condensed notations
\begin{equation*}
\Phi_{(s,t)}(x) := \Phi_{(s,t)}[\mu(s)](x) \qquad \text{and} \qquad \Phi^{\epsilon}_{(s,t)}(x) := \Phi^{\epsilon}_{(s,t)}[\mu_{\epsilon}(s)](x),
\end{equation*}
throughout the remainder of the proof. Recalling that $\mu^0_{\epsilon} = (\Id + \epsilon \F^0)_{\#} \mu^0$ and that the curve  $\mu(\cdot)$ also admits the flow representation \eqref{eq:ExpressionMeasure}, the expression in \eqref{eq:FlowProof0} can be rewritten as
\begin{equation}
\label{eq:FlowProof2}
\mu_{\epsilon}(t) =  \Big( \Phi^{\epsilon}_{(0,t)} \circ \big( \Id + \epsilon \F^0 \big) \circ \Phi_{(t,0)}(\cdot) \Big)_{\raisebox{4pt}{$\scriptstyle{\#}$}} \mu(t),
\end{equation}
for all times $t \in [0,T]$, where ``$\circ$'' is the standard composition operation between functions. Hence, \eqref{eq:PerturbedMeasure_Expr} holds with the family of maps $(\G_{\epsilon}(\cdot,\cdot))$ defined by
\begin{equation}
\label{eq:GepsDef}
\G_{\epsilon} : (t,x) \in [0,T] \times K \mapsto \Phi^{\epsilon}_{(0,t)} \circ \big( \Id + \epsilon \F^0 \big) \circ \Phi_{(t,0)}(x),
\end{equation}
for every $\epsilon \in [-1,1]$. Our goal now is to establish the first-order expansion of \eqref{eq:TaylorMeasure}.

By adapting the argument of \cite[Proposition 5]{PMPWass} to the case of non-local velocity-fields satisfying the localised differentiability assumptions \ref{hyp:H}-$(ii)$ and \ref{hyp:H}-$(iii)$, it can be shown that the map
\begin{equation*}
\epsilon \in [-1,1] \mapsto \Phi^{\epsilon}_{(0,\cdot)}(\cdot) \in C^0([0,T] \times K',\R^d),
\end{equation*}
is Fréchet differentiable at $\epsilon = 0$, where $K':= B(0,R_r')$ and $R_r' > 0$ are such that $(\Id +\epsilon \F^0)(K) \subset K'$ for all $\epsilon \in [-1,1]$. Denoting by $\Psi \in C^0([0,T] \times K',\R^d)$ the corresponding Fréchet derivative, one has
\begin{equation}
\label{eq:TaylorC01}
\Phi^{\epsilon}_{(0,t)}(y) = \Phi_{(0,t)}(y) + \epsilon \Psi(t,y) + o_{t,y}(\epsilon),
\end{equation}
for any $(t,y) \in [0,T] \times K'$, where $\sup_{(t,y) \in [0,T] \times K'} |o_{t,y}(\epsilon)| = o(\epsilon)$. Under hypotheses \ref{hyp:H}-$(i)$ and \ref{hyp:H}-$(ii)$, it is a standard result in the theory of Carathéodory ODEs (see e.g. \cite[Theorem 2.3.2]{BressanPiccoli}) that the flow map $x \in \R^d \mapsto \Phi_{(0,t)}(x) \in \R^d$ is Fréchet-differentiable with
\begin{equation}
\label{eq:TaylorC02}
\Phi_{(0,t)}(x + \epsilon h) = \Phi_{(0,t)}(x) + \epsilon \D_x \Phi_{(0,t)}(x)h + o_{t,x,h}(\epsilon),
\end{equation}
for all times $t \in [0,T]$ and any $x,h \in \R^d$, where
\begin{equation*}
\sup_{(t,x,h) \in [0,T] \times K \times B(0,r')} |o_{t,x,h}(\epsilon)| = o_{r'}(\epsilon),
\end{equation*}
for every $r' > 0$. By merging \eqref{eq:TaylorC01} and \eqref{eq:TaylorC02}, we obtain for all $(t,x) \in [0,T] \times K$ that
\begin{equation}
\label{eq:TaylorC03}
\Phi^{\epsilon}_{(0,t)} \big( x + \epsilon \F^0(x) \big) = \Phi_{(0,t)}(x) + \epsilon \D_x \Phi_{(0,t)}(x) \F^0(x) + \epsilon \Psi(t,x) + o_{t,x}(\epsilon),
\end{equation}
where $\sup_{(t,x) \in [0,T] \times K}|o_{t,x}(\epsilon)| = o(\epsilon)$. We now define the map $\F \in C^0([0,T] \times K,\R^d)$ by
\begin{equation}
\label{eq:Fdef}
\F(t,x) := \D_x \Phi_{(0,t)}(x) \F^0(x) + \Psi(t,x),
\end{equation}
for any $(t,x) \in [0,T] \times K$. By merging \eqref{eq:GepsDef} and \eqref{eq:TaylorC03}, we obtain that $\epsilon \in [-1,1] \mapsto \G_{\epsilon}(t,\cdot) \in C^0(K,\R^d)$ satisfies for all times $t \in [0,T]$ the following Taylor expansion at $\epsilon = 0$
\begin{equation}
\label{eq:F_Taylor}
\G_{\epsilon}(t,\cdot) = \Id + \epsilon \F \big(t, \Phi_{(t,0)}(\cdot) \big) + o_t(\epsilon),
\end{equation}
where $\sup_{t \in [0,T]} \NormC{o_t(\epsilon)}{0}{K,\R^d} = o(\epsilon)$ for every $\epsilon \in [-1,1]$.

We now prove that the map $\F \in C^0([0,T] \times K,\R^d)$ defined in \eqref{eq:Fdef} is the unique solution of the linearised Cauchy problem \eqref{eq:Linearised_CauchyNonLocal}. First, recall that by classical results on differentials of flow maps (see e.g. \cite[Theorem 2.3.2]{BressanPiccoli}), the application $t \in [0,T] \mapsto \D_x \Phi_{(0,t)}(x) \F^0(x) \in \R^d$ is the unique solution of
\begin{equation}
\label{eq:ODE_Exp1}
\D_x \Phi_{(0,t)}(x) \F^0(x) = \F^0(x) + \INTSeg{\D_x v \Big(s,\mu(s), \Phi_{(0,s)}(x) \Big) \Big( \D_x \Phi_{(0,s)}(x) \F^0(x) \Big)}{s}{0}{t},
\end{equation}
for any $x \in K$. By \eqref{eq:TaylorC01}, it further holds
\begin{equation}
\label{eq:ODE_Exp2}
\begin{aligned}
v \Big( s , \mu_{\epsilon}(s) , \Phi^{\epsilon}_{(0,s)}(x) \Big) & = v \Big( s ,  \mu_{\epsilon}(s) , \Phi_{(0,s)}(x) + \epsilon \Psi(s,x) + o_{s,x}(\epsilon) \Big) \\
& = v \Big( s , \mu_{\epsilon}(s) , \Phi_{(0,s)}(x) \Big) + \epsilon \D_x v \Big( s , \mu_{\epsilon}(s) , \Phi_{(0,s)}(x) \Big) \Psi(s,x) + o_{s,x}(\epsilon) \\
& = v \Big( s , \mu_{\epsilon}(s) , \Phi_{(0,s)}(x) \Big) + \epsilon \D_x v \Big( s , \mu(s) , \Phi_{(0,s)}(x) \Big) \Psi(s,x) + o_{s,x}(\epsilon),
\end{aligned}
\end{equation}
for $\Lcal^1$-almost every $s \in [0,T]$ and any $(\epsilon,x) \in [-1,1] \times K'$, with
\begin{equation*}
\INTSeg{\sup_{x \in K'} |o_{s,x}(\epsilon)|}{s}{0}{T} = o(\epsilon),
\end{equation*}
and where we also used that $\nu \in \Pcal(K') \mapsto \D_x v(s,\nu, \Phi_{(0,s)}(x))$ is continuous by \ref{hyp:H}-$(ii)$. Using \ref{hyp:H}-$(iii)$, we can in turn apply the vector-valued chain rule of Remark \ref{rmk:VectVal} together with \eqref{eq:FlowProof2}, \eqref{eq:GepsDef} and \eqref{eq:F_Taylor} to obtain
\begin{equation}
\label{eq:ODE_Exp3}
\begin{aligned}
& v \Big( s,\mu_{\epsilon}(s), \Phi_{(0,s)}(x) \Big) \\
&  = v \bigg( s , \Big( \Id + \epsilon \F (s,\Phi_{(s,0)}(\cdot)) + o_s(\epsilon) \Big)_{\raisebox{4pt}{$\scriptstyle \#$}} \mu(s), \Phi_{(0,s)}(x) \bigg) \\
& = v \Big( s,\mu(s), \Phi_{(0,s)}(x) \Big) + \epsilon \INTDom{\D_{\mu} v \Big( s , \mu(s),\Phi_{(0,s)}(x) \Big) \hspace{-0.035cm} (y) \F \big( s,\Phi_{(s,0)}(y) \big) }{\R^d}{\mu(s)(y)} + o_s(\epsilon) \\
& = v \Big( s,\mu(s), \Phi_{(0,s)}(x) \Big) + \epsilon \INTDom{\D_{\mu} v \Big( s , \mu(s),\Phi_{(0,s)}(x) \Big) \hspace{-0.035cm} \big( \Phi_{(0,s)}(y) \big) \, \F(s,y)}{\R^d}{\mu^0(y)} + o_s(\epsilon), 
\end{aligned}
\end{equation}
for $\Lcal^1$-almost every $s \in [0,T]$ and all $y \in K'$. Observe now that for every $\epsilon \in [-1,1]$, the family of perturbed flows $(\Phi^{\epsilon}_{(0,t)}(\cdot))_{t \in [0,T]}$ can be characterised as the unique solution of the Carathéodory ODE
\begin{equation}
\label{eq:FlowProof1}
\begin{aligned}
\Phi^{\epsilon}_{(0,t)}(y) & = y + \INTSeg{v_{\epsilon} \Big( s , \Phi^{\epsilon}_{(0,s)}(\cdot)_{\#} \mu^0_{\epsilon}, \Phi^{\epsilon}_{(0,s)}(y) \Big)}{s}{0}{t} \\
& = y + \INTSeg{ \bigg( v \Big( s , \mu_{\epsilon}(s) , \Phi^{\epsilon}_{(0,s)}(y) \Big) + \epsilon w \Big( s , \Phi^{\epsilon}_{(0,s)}(y) \Big) \bigg)}{s}{0}{t},
\end{aligned}
\end{equation}
for any $(t,y) \in [0,T] \times K'$. Therefore by plugging \eqref{eq:TaylorC01}, \eqref{eq:ODE_Exp2} and \eqref{eq:ODE_Exp3} into \eqref{eq:FlowProof1}, we further obtain
\begin{equation*}
\begin{aligned}
\Phi_{(0,t)}(x) + \epsilon \Psi(t,x) = x  & + \INTSeg{v \Big( s,\mu(s) , \Phi_{(0,s)}(x) \Big)}{s}{0}{t} + \epsilon \INTSeg{\D_x v \Big( s,\mu(s), \Phi_{(0,s)}(x) \Big) \Psi(s,x)}{s}{0}{t} \\
& + \epsilon \INTSeg{ \INTDom{\D_{\mu} v \Big( s,\mu(s), \Phi_{(0,s)}(x) \Big) \big( \Phi_{(0,s)}(y) \big) \, \F(s,y)}{\R^d}{\mu^0(y)}}{s}{0}{t} \\
& + \epsilon \INTSeg{w \Big( s, \Phi_{(0,s)}(x) \Big)}{s}{0}{t} + o_{t,x}(\epsilon),
\end{aligned}
\end{equation*}
for any $(t,x) \in [0,T] \times K$ and all $\epsilon \in [-1,1]$, where $\sup_{(t,x) \in [0,T] \times K} |o_{t,x}(\epsilon)| = o(\epsilon)$. By identifying terms and observing that the flow maps $(\Phi_{(0,t)}(\cdot))_{t \in [0,T]}$ satisfy \eqref{eq:FlowExp1}, we recover the ODE characterisation
\begin{equation}
\label{eq:ODE_Exp4}
\begin{aligned}
\Psi(t,x) & = \INTSeg{\bigg( \D_x v \Big( s,\mu(s), \Phi_{(0,s)}(x) \Big) \Psi(s,x) + w \Big( s , \Phi_{(0,s)}(x) \Big) \bigg)}{s}{0}{t}  \\
& \hspace{0.4cm} + \INTSeg{ \INTDom{\D_{\mu} v \Big( s,\mu(s), \Phi_{(0,s)}(x) \Big) \big( \Phi_{(0,s)}(y) \big) \, \F(s,y)}{\R^d}{\mu^0(y)}}{s}{0}{t},
\end{aligned}
\end{equation}
for all $(t,x) \in [0,T] \times K$. Upon combining \eqref{eq:ODE_Exp1} and \eqref{eq:ODE_Exp4}, we conclude that the map $\F \in C^0([0,T] \times K,\R^d)$ defined in \eqref{eq:Fdef} solves \eqref{eq:Linearised_CauchyNonLocal}. The uniqueness of solutions to this linearised problem can then be recovered under hypotheses \ref{hyp:H} from Gr\"onwall's Lemma.
\end{proof}

In the following theorem, we show that when the velocity perturbation $w : [0,T] \times \R^d \rightarrow \R^d$ belongs to the tangent cone to the admissible velocities of a continuity inclusion of the form \eqref{eq:WassIncDef} with a convexified right-hand side, there exist solutions of the non-convexified inclusion whose distance to the measure defined in \eqref{eq:PerturbedMeasure_Expr} is infinitesimally small compared to $\epsilon > 0$.

\begin{Def}[Tangent cones to convex hulls of $C^0(\R^d,\R^d)$-valued multifunctions]
\label{def:AdjacentCone_C0}
Let $V : [0,T] \times \Pcal_c(\R^d) \rightrightarrows C^0(\R^d,\R^d)$ be a set-valued map and $v \in \co V(t,\mu)$ for some $t \in [0,T]$ and $\mu \in \Pcal_c(\R^d)$, where $\co V(t,\mu)$ is given in the sense of \eqref{eq:ConvbarDef}. The \textnormal{tangent cone} to $\co V(t,\mu)$ at $v$ is defined by
\begin{equation*}
T_{\co V(t,\mu)}(v) := \bigg\{ w \in C^0(\R^d,\R^d) ~\text{s.t.}~ w_{\vert K} \in T_{\co V_K(t,\mu)}(v_{\vert K}) ~ \text{for any $K := B(0,R)$ with $R > 0$} \bigg\},
\end{equation*}
where $T_{\co V_K(t,\mu)}(v_{\vert K})$ is taken in the Banach space $\big(C^0(K,\R^d),\NormC{\cdot}{0}{K,\R^d} \big)$.
\end{Def}

\begin{thm}[Variational linearisation]
\label{thm:Variational_Inclusion}
Let $\mu^0 \in \Pcal(B(0,r))$ for some $r > 0$ and $V : [0,T] \times \Pcal_c(\R^d) \rightrightarrows C^0(\R^d,\R^d)$ be a set-valued map satisfying hypotheses \textnormal{\ref{hyp:DI}} with $p \in [1,+\infty)$. Moreover, let $(\mu(\cdot),\vb(\cdot))$ be a trajectory-selection pair solution of the inclusion \eqref{eq:WassIncDef}, and suppose that for $\Lcal^1$-almost every $t \in [0,T]$ the following representation
\begin{equation*}
\vb(t) := v(t,\mu(t)) \in V(t,\mu(t)),
\end{equation*}
holds with $v : [0,T] \times \Pcal_c(\R^d) \times \R^d \rightarrow \R^d$ satisfying hypotheses \ref{hyp:H} and $v(t,\nu) \in V(t,\nu)$ for $\Lcal^1$-almost every $t \in [0,T]$ and any $\nu \in \Pcal_c(\R^d)$. Let $\F^0 \in C^0(\R^d,\R^d)$, $w : [0,T] \times \R^d \rightarrow \R^d$ be a Carathéodory vector-field satisfying hypotheses \ref{hyp:CE} along with the pointwise inclusions
\begin{equation}
\label{eq:Tangent_Inc}
w(t) \in T_{\co V(t,\mu(t))} \big( v(t,\mu(t)) \big) \qquad \text{for $\Lcal^1$-almost every $t \in [0,T]$},
\end{equation}
and $\F : [0,T] \times \R^d \rightarrow \R^d$ be the corresponding solution of the linearised Cauchy problem \eqref{eq:Linearised_CauchyNonLocal}.

Then for any $\epsilon \in [0,1]$, there exists a solution $\tilde{\mu}_{\epsilon}(\cdot)$ of the continuity inclusion
\begin{equation}
\label{eq:VarInc_Diff}
\left\{
\begin{aligned}
& \partial_t \tilde{\mu}_{\epsilon}(t) \in - \Div \Big( V(t,\tilde{\mu}_{\epsilon}(t)) \tilde{\mu}_{\epsilon}(t) \Big), \\
& \tilde{\mu}_{\epsilon}(0) = (\Id + \epsilon \F^0)_{\#} \mu^0,
\end{aligned}
\right.
\end{equation}
such that
\begin{equation}
\label{eq:VarInc_Est}
\sup_{t \in [0,T]} W_p \Big( \tilde{\mu}_{\epsilon}(t), \G_{\epsilon}(t,\cdot)_{\#} \mu(t) \Big) = o(\epsilon),
\end{equation}
where the family of maps $(\G_{\epsilon}(\cdot,\cdot)) \subset C^0([0,T] \times \R^d,\R^d)$ is given as in Proposition \ref{prop:Linearisation_Cauchy}.
\end{thm}

\begin{proof}
In the sequel, we shall interchangeably use the notations $v(t,\mu(t))$ and $w(t)$ to denote the vector-fields defined on $\R^d$ or their restrictions to some closed ball $K := B(0,R)$.

First, remark that $v : [0,T] \times \Pcal_c(\R^d) \times \R^d \mapsto \R^d$ satisfies hypotheses \ref{hyp:CE} as a direct consequence of \ref{hyp:H}-$(i)$. Hence, by Corollary \ref{cor:NonLocalPDE}, there exists $R_r > 0$ such that $\supp(\mu(t)) \subset K := B(0,R_r)$ for all times $t \in [0,T]$. Furthermore observe that since $w(\cdot)$ satisfies \eqref{eq:Tangent_Inc}, it holds by Definition \ref{def:AdjacentConeBanach} that
\begin{equation}
\label{eq:Proof_VarInc1}
\dist_{C^0(K,\R^d)} \Big( v(t,\mu(t)) + \epsilon w(t) ; \co V_K(t,\mu(t)) \Big) = o_t(\epsilon),
\end{equation}
for $\Lcal^1$-almost every $t \in [0,T]$ and all $\epsilon \in [0,1]$, where $o_t(\epsilon)/\epsilon \rightarrow 0$ as $\epsilon \rightarrow 0^+$. Moreover, one also has 
\begin{equation}
\label{eq:Proof_VarInc1Bis}
\dist_{C^0(K,\R^d)} \Big( v(t,\mu(t)) + \epsilon w(t) ; \co V_K(t,\mu(t)) \Big) \leq  \epsilon \NormC{w(t)}{0}{K,\R^d},
\end{equation}
because $v(t,\mu(t)) \in V(t,\mu(t))$, which implies in particular that the remainder defined in \eqref{eq:Proof_VarInc1} divided by $\epsilon > 0$ is integrably bounded. In addition, observe that the non-local velocity field
\begin{equation*}
v_{\epsilon} : (t,\mu,x) \in [0,T] \times \Pcal_c(\R^d) \times \R^d \mapsto v(t,\mu,x) + \epsilon w(t,x),
\end{equation*}
also satisfies hypotheses \ref{hyp:CE} for every $\epsilon \in [0,1]$ by construction.

Denote by $\mu_{\epsilon}(\cdot) \in \AC([0,T],\Pcal_c(\R^d))$ the unique solution of the perturbed Cauchy problem
\begin{equation}
\label{eq:Proof_PerturbedCauchy}
\left\{
\begin{aligned}
& \partial_t \mu_{\epsilon}(t) + \Div \Big( \big( v(t,\mu_{\epsilon}(t)) + \epsilon w(t) \big) \mu_{\epsilon}(t) \Big) = 0, \\
& \mu_{\epsilon}(0) = (\Id + \epsilon \F^0)_{\#} \mu^0,
\end{aligned}
\right.
\end{equation}
and let $K' := B(0,R_r')$ be a closed ball such that $\supp(\mu_{\epsilon}(t)) \subset K'$ for any $(t,\epsilon) \in [0,T] \times [0,1]$. By Theorem \ref{thm:FilippovEstimate} applied to $V(\cdot,\cdot) := \{ v(\cdot,\cdot) \}$ together with Corollary \ref{cor:NonLocalPDE}, there exists $C>0$ such that
\begin{equation}
\label{eq:Proof_VarInc2}
\sup_{t \in [0,T]} W_p(\mu(t),\mu_{\epsilon}(t)) \leq C \epsilon,
\end{equation}
for every $\epsilon \in [0,1]$. Since $w(\cdot)$ satisfies \eqref{eq:Proof_VarInc1} and \eqref{eq:Proof_VarInc1Bis}, we can apply \cite[Theorem 8.1.3]{Aubin1990} to recover the existence of a family $(w_{\epsilon}(\cdot))$ of $\Lcal^1$-measurable maps such that 
\begin{equation}
\label{eq:Proof_VarInc2Bis}
\left\{
\begin{aligned}
& \NormC{w(t) - w_{\epsilon}(t)}{0}{K',\R^d} \underset{\epsilon \rightarrow 0^+}{\longrightarrow} 0 \quad \text{and} \quad \NormC{w_{\epsilon}(t)}{0}{K',\R^d} \leq 2 \NormC{w(t)}{0}{K',\R^d}+1,\\
& \, v(t,\mu(t)) + \sqrt{\epsilon} w_{\epsilon}(t) \in \co V_{K'}(t,\mu(t)) ,
\end{aligned}
\right.
\end{equation}
for $\Lcal^1$-almost every $t \in [0,T]$ and all $\epsilon \in [0,1]$. By hypothesis \ref{hyp:DI}-$(iv)$ together with \eqref{eq:Proof_VarInc2} and \cite[Theorem 8.1.3]{Aubin1990}, one can also find measurable selections $t \in [0,T] \mapsto \vb_{\epsilon}(t) \in \co V_{K'}(t,\mu_{\epsilon}(t))$ such that
\begin{equation}
\label{eq:Proof_VarInc3}
\NormC{v(t,\mu(t)) + \sqrt{\epsilon} w_{\epsilon}(t) - \vb_{\epsilon}(t)}{0}{K',\R^d} \leq L_{K'}(t) W_p(\mu(t),\mu_{\epsilon}(t)) \leq L_{K'}(t) C \epsilon,
\end{equation}
for $\Lcal^1$-almost every $t \in [0,T]$ and any $\epsilon \in [0,1]$. Moreover observing that the sets $\co V_{K'}(t,\mu_{\epsilon}(t)) \subset C^0(K',\R^d)$ are convex, it holds
\begin{equation*}
(1-\sqrt{\epsilon}) v(t,\mu_{\epsilon}(t)) + \sqrt{\epsilon} \vb_{\epsilon}(t) \in \co V_{K'}(t,\mu_{\epsilon}(t)),
\end{equation*}
which along with \eqref{eq:Proof_VarInc3} further yields
\begin{equation}
\label{eq:Proof_VarInc4}
(1-\sqrt{\epsilon}) v(t,\mu_{\epsilon}(t)) + \sqrt{\epsilon} \big( v(t,\mu(t)) + \sqrt{\epsilon} w_{\epsilon}(t) \big) \in \co V_{K'}(t,\mu_{\epsilon}(t)) + L_{K'}(t) C \epsilon^{3/2} \, \B_{C^0(K',\R^d)},
\end{equation}
for $\Lcal^1$-almost every $t \in [0,T]$ and any $\epsilon \in [0,1]$. Recalling that $v : [0,T] \times \Pcal_c(\R^d) \times \R^d \rightarrow \R^d$ satisfies hypotheses \ref{hyp:CE}, we again have by \eqref{eq:Proof_VarInc2} 
\begin{equation*}
\NormC{v(t,\mu(t)) - v(t,\mu_{\epsilon}(t))}{0}{K',\R^d} ~\leq~ L_{K'}(t) W_p(\mu(t),\mu_{\epsilon}(t)) ~\leq~ L_{K'}(t) C \epsilon,
\end{equation*}
for $\Lcal^1$-almost every $t \in [0,T]$. Upon plugging this last estimate into \eqref{eq:Proof_VarInc4}, we obtain
\begin{equation*}
\begin{aligned}
v(t,\mu_{\epsilon}(t)) + \epsilon w(t) & = (1-\sqrt{\epsilon}) v(t,\mu_{\epsilon}(t)) + \sqrt{\epsilon} \big( v(t,\mu_{\epsilon}(t)) + \sqrt{\epsilon} w_{\epsilon}(t) \big) + \epsilon \big( w(t) - w_{\epsilon}(t) \big) \\
& \in \co V_{K'}(t,\mu_{\epsilon}(t)) + \Big( 2 L_{K'}(t) C \epsilon^{3/2} +  \epsilon \NormC{w(t)-w_{\epsilon}(t)}{0}{K',\R^d} \Big) \B_{C^0(K',\R^d)}.
\end{aligned}
\end{equation*}
This together with the properties \eqref{eq:Proof_VarInc2Bis} of the family of maps  $(w_{\epsilon}(\cdot))$ and the fact that $L_{K'}(\cdot) \in L^1([0,T],\R_+)$ imply that for every $\epsilon \in [0,1]$, the mismatch function
\begin{equation}
\eta_{\epsilon} : t \in [0,T] \mapsto \dist_{C^0(K',\R^d)} \Big( v(t,\mu_{\epsilon}(t)) + \epsilon w(t) ; \co V_{K'}(t,\mu_{\epsilon}(t)) \Big),
\end{equation}
is integrable and satisfies $r(\epsilon) := \Norm{\eta_{\epsilon}(\cdot)}_1 = o(\epsilon)$ for all $\epsilon \in [0,1]$ by Lebesgue's dominated convergence theorem. Whence, we can apply Theorem \ref{thm:FilippovEstimate} to obtain the existence of a curve $\tilde{\mu}_{\epsilon}(\cdot)$ solution of
\begin{equation}
\label{eq:Proof_RelaxedInc}
\left\{
\begin{aligned}
& \partial_t \tilde{\mu}_{\epsilon}(t) \in - \Div \Big( \co V(t,\tilde{\mu}_{\epsilon}(t)) \tilde{\mu}_{\epsilon}(t) \Big), \\
& \tilde{\mu}_{\epsilon}(0) = (\Id + \epsilon \F^0)_{\#} \mu^0,
\end{aligned}
\right.
\end{equation}
such that
\begin{equation}
\label{eq:Proof_DistEst}
\sup_{t \in [0,T]} W_p(\tilde{\mu}_{\epsilon}(t),\mu_{\epsilon}(t)) = o(\epsilon),
\end{equation}
for any $\epsilon \in [0,1]$. Moreover by Theorem \ref{thm:RelaxationWass} applied with $\delta := r(\epsilon)$, we can choose the curve $\tilde{\mu}_{\epsilon}(\cdot)$ satisfying \eqref{eq:Proof_DistEst} as a solution of \eqref{eq:VarInc_Diff} instead of \eqref{eq:Proof_RelaxedInc}. Finally, observe that by Proposition \ref{prop:Linearisation_Cauchy} one can express $\mu_{\epsilon}(t)$ for all $t \in [0,T]$ as
\begin{equation}
\label{eq:Proof_mupert}
\mu_{\epsilon}(t) = \G_{\epsilon}(t,\cdot)_{\#} \mu(t),
\end{equation}
where the family of maps $(\G_{\epsilon}(\cdot,\cdot)) \subset C^0([0,T] \times K,\R^d)$ satisfies \eqref{eq:TaylorMeasure}. This together with \eqref{eq:Proof_DistEst} yields
\begin{equation*}
\sup_{t \in [0,T]} W_p(\tilde{\mu}_{\epsilon}(t),\G_{\epsilon}(t,\cdot)_{\#} \mu(t)) = o(\epsilon),
\end{equation*}
which concludes the proof of Theorem \ref{thm:Variational_Inclusion}.
\end{proof}


\section{Pontryagin Maximum Principle in Wasserstein spaces}
\label{section:PMP}

In this section, we apply the differential-theoretic concepts studied in Section \ref{section:NonSmoothWass} to prove a \textit{Pontryagin Maximum Principle} (``PMP'' in the sequel) for optimal control problems in Wasserstein spaces with inequality final-point constraints. This result partially improves those of \cite{PMPWassConst,PMPWass} as the strategy is more concise and allows to discriminate between the normal and abnormal scenarios of the PMP.

To lighten the exposition, we will illustrate the full proof strategy on a \textit{Mayer} problem in Section \ref{subsection:Mayer}, and we will then prove the PMP for a \textit{Bolza} problem in Section \ref{subsection:Bolza} by adapting a standard procedure which directly builds on the PMP for the Mayer problem.


\subsection{The PMP for constrained Mayer problems in Wasserstein spaces}
\label{subsection:Mayer}

In this section, we focus our attention on the following constrained Mayer problem
\begin{equation*}
(\Ppazo_\MC) ~ \left\{
\begin{aligned}
\min_{u(\cdot) \in \U} & \Big[ \varphi(\mu(T)) \Big] \\
\text{s.t.}~ & \left\{
\begin{aligned}
& \partial_t \mu(t) + \Div \Big( v (t,\mu(t),u(t)) \mu(t) \Big) = 0, \\
& \mu(0) \, = \mu^0, \\
& \mu(T) \in \Qpazo_T, \\
\end{aligned}
\right.
\end{aligned}
\right.
\end{equation*}
where $\varphi : \Pcal_c(\R^d) \rightarrow \R$ is a given final cost. Here, we fix an initial datum $\mu^0 \in \Pcal_c(\R^d)$ and a controlled non-local velocity field $v : [0,T] \times \Pcal_c(\R^d) \times U \times \R^d \rightarrow \R^d$. The minimisation in $(\Ppazo_{\MC})$ is taken over the set of admissible open-loop controls
\begin{equation*}
\U := \Big\{ u : [0,T] \rightarrow U ~\text{s.t. $u(\cdot)$ is $\Lcal^1$-measurable} \Big\},
\end{equation*}
where $(U,d_U)$ is a compact metric space, and we suppose that the set of \textit{final-point constraints} $\Qpazo_T$ is defined by a collection of functional inequalities of the form
\begin{equation*}
\Qpazo_T := \Big\{ \mu \in \Pcal_2(\R^d) ~\text{s.t.}~ \Psi_i(\mu) \leq 0 ~ \text{for all $i \in \{1,\dots,n\}$} \Big\},
\end{equation*}
where $\Psi_i : \Pcal_c(\R^d) \rightarrow \R$ for every index $i \in \{ 1,\dots,n\}$. From now on, we consider $\Pcal_c(\R^d)$ as a subset of the metric space $(\Pcal_1(\R^d),W_1)$.

\begin{taggedhyp}{\textbn{(MCP)}}
\label{hyp:MCP}
For every $R >0$, assume that the following holds with $K := B(0,R)$. 
\begin{enumerate}
\item[$(i)$] For every $u \in U$, the non-local velocity-field $(t,\mu,x) \in [0,T] \times \Pcal_c(\R^d) \times \R^d \mapsto v(t,\mu,u)(x) \in \R^d$ satisfies hypotheses \ref{hyp:H} with $p=1$ and constants that are independent of $u \in U$. Moreover, the map $u \in U \mapsto v(t,\mu,u)(x) \in \R^d$ is continuous for $\Lcal^1$-almost every $t \in [0,T]$ and any $(\mu,x) \in \Pcal_c(\R^d) \times \R^d$.
\item[$(ii)$] The final cost $\varphi(\cdot)$ and the constraint functionals $\{\Psi_i(\cdot) \}_{1 \leq i \leq n}$ are Lipschitz continuous in the $W_1$-metric over $\Pcal(K)$ and locally differentiable over $\Pcal_c(\R^d)$. Moreover, the maps
\begin{equation*}
x \in \R^d \mapsto \nabla \varphi(\mu)(x) \in \R^d \qquad \text{and} \qquad x \in \R^d \mapsto \nabla \Psi_i(\mu)(x) \in \R^d,
\end{equation*}
are continuous for every $i \in \{1,\dots,n\}$.
\end{enumerate}
\end{taggedhyp}

As illustrated in \cite{ContInc} and recalled in Theorem \ref{thm:ControlInc} above, the set of all trajectories $\mu(\cdot)$ satisfying
\begin{equation}
\label{eq:NonLocalBis}
\qquad \quad \partial_t \mu(t) + \Div \big( v(t,\mu(t),u(t)) \mu(t) \big) = 0 ~~ \text{for some $u(\cdot) \in \U$}, 
\end{equation}
coincides exactly with the solution set of the continuity inclusion
\begin{equation}
\label{eq:DiffBis}
\partial_t \mu(t) \in - \Div \Big( V(t,\mu(t)) \mu(t) \Big),
\end{equation}
when the set-valued map $V : [0,T] \times \Pcal_c(\R^d) \rightrightarrows C^0(\R^d,\R^d)$ is defined by
\begin{equation}
\label{eq:SetValued_Def}
V (t,\mu) := \Big\{ \vb \in C^0(\R^d,\R^d) ~\text{s.t.}~ \vb(\cdot) = v(t,\mu,u,\cdot)~ \text{for some $u \in U$} \Big\},
\end{equation}
for all $(t,\mu) \in [0,T] \times \Pcal_c(\R^d)$.

\begin{Def}[Admissible pairs and strong local minimisers for $(\Ppazo_{\MC})$]
We say that $(\mu(\cdot),u(\cdot))$ is an \textnormal{admissible trajectory-control pair} for $(\Ppazo_{\MC})$ if $u(\cdot) \in \U$ and $\mu(\cdot) \in \AC([0,T],\Pcal_c(\R^d))$ is a solution of the controlled non-local continuity equation \eqref{eq:NonLocalBis} satisfying $\mu(0) = \mu^0$ and $\mu(T) \in \Qpazo_T$. Moreover, an admissible pair $(\mu^*(\cdot),u^*(\cdot))$ is a \textnormal{strong local minimiser} for $(\Ppazo_{\MC})$ if there exists $\epsilon > 0$ such that
\begin{equation*}
\varphi(\mu^*(T)) \leq \varphi(\mu(T)),
\end{equation*}
for every other admissible pair $(\mu(\cdot),u(\cdot))$ which satisfies $\sup_{t \in [0,T]} W_1(\mu^*(t),\mu(t)) \leq \epsilon$.
\end{Def}

\begin{rmk}[Strong local $W_p$-minimisers]
Observe that if $\mu,\nu \in \Pcal(B(0,R))$ for some $R > 0$, it holds that $W_1(\mu,\nu) \leq W_p(\mu,\nu) \leq (2R)^{(p-1)/p} W_1(\mu,\nu)^{1/p}$ for every $p \in [1,+\infty)$. Hence, an admissible pair is a strong local $W_p$-minimiser if and only if it is a strong local $W_1$-minimiser. 
\end{rmk}

\begin{rmk}[On the existence of optimal trajectory-control pairs]
In \cite[Remark 4 and Theorem 7]{ContInc}, an existence result is provided for a more general variant of problem $(\Ppazo_{\MC})$ in the presence of running and final-point constraints, under a set of hypotheses which are a natural adaptation of \ref{hyp:MCP}. The only additional requirement needed to obtain this existence result compared to necessary optimality conditions is the convexity of the sets of admissible velocities $V(t,\mu) \subset C^0(\R^d,\R^d)$, which is instrumental in showing that the solution set of the corresponding inclusion is compact in the topology of the uniform convergence (see \cite[Theorem 6]{ContInc}). We would like to stress that in \cite[Theorem 7]{ContInc}, it is also assumed that $u \in U \mapsto v(t,\mu,u,x)$ is Lipschitz continuous. However, this hypothesis is only needed when the controls are closed-loop, and it follows from our proof therein that the usual continuity assumption is sufficient when the controls are open-loop, see \cite[Remark 5]{ContInc}. 
\end{rmk}

Let $r > 0$ be such that $\mu^0 \in \Pcal(B(0,r))$, and observe that under hypotheses \ref{hyp:MCP} the non-local velocity-fields $(t,\mu,x) \in [0,T] \times \Pcal_c(\R^d) \times \R^d \mapsto v(t,\mu,u,x)$ satisfy hypotheses \ref{hyp:CE} with constants which are uniform with respect to $u \in U$. Therefore by Theorem \ref{thm:ExistenceWass}, there exists a radius $R_r > 0$ depending only on the magnitudes of $r,\Norm{m(\cdot)}_1$ such that every admissible pair $(\mu(\cdot),u(\cdot))$ verifies
\begin{equation*}
\supp(\mu(t)) \subset K := B(0,R_r),
\end{equation*}
for all times $t \in [0,T]$. In the sequel, we will denote by $(\Phi^{u^*}_{(s,t)}(\cdot))_{s,t \in [0,T]}$ the semigroup of non-local flows which represent the optimal curve $\mu^*(\cdot)$ via \eqref{eq:ExpressionMeasure} and \eqref{eq:FlowExp1}.

We are now ready to state and prove our main result, which is a PMP for $(\Ppazo_{\MC})$. In what follows, we denote by $\J_{2d}$ the \textit{symplectic matrix} of $\R^{2d}$, i.e.
\begin{equation*}
\J_{2d} = \begin{pmatrix} 0 && \Id \\ - \Id && 0 \end{pmatrix},
\end{equation*}
and by $\H : [0,T] \times \Pcal_c(\R^{2d}) \times U \rightarrow \R$ the \textit{Hamiltonian} of the control problem, defined by
\begin{equation}
\label{eq:Thm_HamiltonianDef}
\H(t,\nu,u) = \INTDom{\big\langle r , v(t,\pi^1_{\#} \nu,u,x) \big\rangle}{\R^{2d}}{\nu(x,r)},
\end{equation}
for all $(t,\nu,u) \in [0,T] \times \Pcal_c(\R^{2d}) \times U$. We also consider the set of \textit{active indices} at $\mu^*(T)$, given by
\begin{equation*}
I^{\circ}(\mu^*(T)) := \Big\{ i \in \{1,\dots,n\} ~\text{s.t.}~ \mu^*(T) \in \partial \Qpazo_T^i \Big\},
\end{equation*}
where we introduced the sets $\Qpazo_T^i := \big\{ \mu \in \Pcal_2(\R^d) ~\text{s.t.}~ \Psi_i(\mu) \leq 0 \big\}$ for every $i \in \{1,\dots,n \}$.

\begin{thm}[Pontryagin Maximum Principle for $(\Ppazo_{\MC})$]
\label{thm:PMPMayer}
Let $(\mu^*(\cdot),u^*(\cdot)) \in \AC([0,T],\Pcal_c(\R^d)) \times \U$ be a strong local minimiser for $(\Ppazo_{\MC})$, and suppose that hypotheses \textnormal{\ref{hyp:MCP}} hold.

Then, there exist $R_r' > 0$, non-trivial Lagrange multipliers $(\lambda_0,\lambda_1,\dots,\lambda_n) \in \{ 0,1\} \times \R_+^n$ and a curve of measures $\nu^*(\cdot) \in \AC([0,T],\Pcal(K' \times K'))$ with $K' := B(0,R_r')$ such that the following holds.
\begin{enumerate}
\item[$(i)$] The curve $\nu^*(\cdot)$ solves the \textnormal{forward-backward Hamiltonian} continuity equation
\begin{equation}
\label{eq:Thm_HamiltonianSys}
\left\{
\begin{aligned}
& \partial_t \nu^*(t) + \Div \Big( \J_{2d} \nabla_{\nu} \H(t,\nu^*(t),u^*(t)) \nu^*(t) \Big) = 0, \\
& \pi^1_{\#} \nu^*(t) = \mu^*(t) \hspace{1.5cm} \text{for all times $t \in [0,T]$}, \\
& \nu^*(T) = \bigg( \Id \, , \, \Big( - \lambda_0 \nabla \varphi(\mu^*(T)) - \sum_{i=1}^n \lambda_i \nabla \Psi_i(\mu^*(T)) \Big) \bigg)_{\raisebox{6pt}{$\scriptstyle \#$}} \mu^*(T),
\end{aligned}
\right.
\end{equation}
where the Wasserstein gradient of the Hamiltonian is given explicitly by 
\begin{equation}
\label{eq:PMP_HamiltonianGrad}
\begin{aligned}
& \nabla_{\nu} \H(t,\nu^*(t),u^*(t))(x,r) \\
& \hspace{1cm} = \begin{pmatrix}
\, \D_x v \big( t,\mu^*(t),u^*(t),x \big)^{\top} r + \INTDom{\D_{\mu} v \big( t,\mu^*(t),u^*(t),y \big)(x)^{\top} p \,}{\R^{2d}}{\nu^*(t)(y,p)} \\ \\
v \big( t,\mu^*(t),u^*(t),x \big)
\end{pmatrix},
\end{aligned}
\end{equation}
for $\Lcal^1$-almost every $t \in [0,T]$ and any $(x,r) \in \R^{2d}$.
\item[$(ii)$] The \textnormal{complementarity slackness} conditions 
\begin{equation}
\label{eq:PMP_Complementarity}
\lambda_i \Psi_i(\mu^*(T)) = 0, 
\end{equation}
hold for every index $i \in \{1,\dots,n\}$. 
\item[$(iii)$] The \textnormal{Pontryagin maximisation condition}
\begin{equation}
\label{eq:Thm_Maximisation}
\H(t,\nu^*(t),u^*(t)) = \max_{u \in U} \, \H(t,\nu^*(t),u),
\end{equation}
holds for $\Lcal^1$-almost every $t \in [0,T]$.
\end{enumerate}
Moreover, if there exists a map $t \in [0,T] \mapsto w(t) \in T_{\co V(t,\mu^*(t))} \big( v(t,\mu^*(t),u^*(t) \big)$ satisfying hypotheses \ref{hyp:CE} such that the corresponding solution $\F \in C^0([0,T]\times K,\R^d)$ of \eqref{eq:Linearised_CauchyNonLocal} with $\F^0 = 0$ satisfies
\begin{equation*}
\langle \nabla \Psi_i(\mu^*(T)) , \F \big( T,\Phi^{u^*}_{(T,0)}(\cdot) \big) \rangle_{L^2(\mu^*(T))} < 0,
\end{equation*}
for every $i \in I^{\circ}(\mu^*(T))$, then the PMP is normal, i.e. $\lambda_0 = 1$.
\end{thm}

We split the proof of Theorem \ref{thm:PMPMayer} into four steps. In Step 1, we introduce suitable variational linearisations of the controlled non-local continuity equation inspired by Theorem \ref{thm:Variational_Inclusion}. In Step 2, we focus on the situation in which the end-points of the trajectories of the linearised system do not satisfy the constraint qualification condition, which leads to the \textit{abnormal} PMP, i.e. the case $\lambda_0 = 0$. In Step 3, we analyse the converse scenario and prove that it corresponds to the \textit{normal} PMP, i.e. the case $\lambda_0 = 1$. Finally in Step 4, we build a state-costate curve $\nu^*(\cdot)$ solution of the Hamiltonian flow \eqref{eq:Thm_HamiltonianSys}, along which the maximum principle \eqref{eq:Thm_Maximisation} holds.


\paragraph*{Step 1: Variational linearisations along $(\mu^*(\cdot),u^*(\cdot))$.} Observe that $v(t,\mu,u^*(t)) \in \co V(t,\mu)$ for $\Lcal^1$-almost every $t \in [0,T]$ and all $\mu \in \Pcal_c(\R^d)$, where the set-valued map $V : [0,T] \times \Pcal_c(\R^d) \rightrightarrows C^0(\R^d,\R^d)$ is given by \eqref{eq:SetValued_Def}, and satisfies hypotheses \ref{hyp:DI} as a consequence of \ref{hyp:MCP}-$(i)$. Let now
\begin{equation}
\label{eq:w_def}
t \in [0,T] \mapsto w(t) \in T_{\co V(t,\mu^*(t))} \big( v(t,\mu^*(t),u^*(t)) \big),
\end{equation}
be any $\Lcal^1$-measurable selection satisfying hypotheses \ref{hyp:CE}, where $T_{\co V(t,\mu^*(t))} \big(v(t,\mu^*(t),u^*(t) \big)$ is to be understood for $\Lcal^1$-almost every $t \in [0,T]$ in the sense of Definition \ref{def:AdjacentCone_C0}. Since the non-local velocity-field $(t,\mu,x) \mapsto v(t,\mu,u^*(t),x)$ satisfies hypotheses \ref{hyp:H}, we can apply Theorem \ref{thm:Variational_Inclusion} to obtain the existence of a solution $\tilde{\mu}_{\epsilon}(\cdot)$ of \eqref{eq:DiffBis} with $\tilde{\mu}_{\epsilon}(0) = \mu^0$, such that
\begin{equation}
\label{eq:DisEst_Step1}
\sup_{t \in [0,T]} W_1 \Big( \tilde{\mu}_{\epsilon}(t) , \big( \Id + \epsilon \F \big( t , \Phi^{u^*}_{(t,0)}(\cdot) \big) + o_t(\epsilon) \big)_{\#} \mu^*(t) \Big) = o(\epsilon),
\end{equation}
for any $\epsilon > 0$ and all times $t \in [0,T]$, where $\sup_{t \in [0,T]} \NormC{o_t(\epsilon)}{0}{K,\R^d} = o(\epsilon)$. Here, the map $\F \in C^0([0,T] \times K,\R^d)$ is the unique solution of \eqref{eq:Linearised_CauchyNonLocal} with $\F^0 = 0$ and $w(\cdot)$ satisfying \eqref{eq:w_def} and \ref{hyp:CE}. From now on, we will denote by $\Rpazo^L_T \subset C^0(\R^d,\R^d)$ the \textit{reachable set at time $T$} of \eqref{eq:Linearised_CauchyNonLocal} with $\F^0 = 0$, where the non-local velocity-field
\begin{equation*}
(t,x) \in [0,T] \times \R^d \mapsto v(t,\mu(t),\Phi_{(0,t)}[\mu^0](x)) \in \R^d,
\end{equation*}
is replaced by the controlled non-local vector-field
\begin{equation*}
(t,x) \in [0,T] \times \R^d \mapsto v(t,\mu^*(t),u^*(t),\Phi^{u^*}_{(0,t)}(x)) \in \R^d,
\end{equation*}
namely
\begin{equation*}
\begin{aligned}
\Rpazo_T^L := \bigg\{ \G_T \in C^0(\R^d,\R^d) ~\text{s.t.}~ \G_T(\cdot) = \F \big( T , \Phi_{(T,0)}^{u^*}(\cdot) \big) ~\text{where $\F(\cdot,\cdot)$ solves \eqref{eq:Linearised_CauchyNonLocal} with $\F^0 = 0$} & \\
\text{and $w(\cdot)$ satisfying \eqref{eq:w_def} and hypotheses \ref{hyp:CE}} \, & \bigg\} .
\end{aligned}
\end{equation*}

Up to relabelling the functionals $\{ \Psi_i(\cdot)\}_{i=1}^n$, we can suppose without loss of generality that $I^{\circ}(\mu^*(T)) = \{1,\dots,k\}$ for some $k \leq n$, whenever it is non-empty. In addition, observe that if $\nabla \Psi_i(\mu^*(T)) = 0$ for some $i \in I^{\circ}(\mu^*(T))$, then the statements of the PMP are verified with $\lambda_i = 1$,  $\lambda_j = 0$ for every $j \in \{0,\dots,n \} \backslash \{i\}$ and $\nu^*(t) := \mu^*(t) \times \delta_0$ for all times $t \in [0,T]$. In this case, one has
\begin{equation*}
\H(t,\nu^*(t),u) \equiv 0, \quad \nabla_{\nu} \H(t,\nu^*(t),u^*(t))(x,r) = \begin{pmatrix}
0 \\ v(t,\mu^*(t),u^*(t),x)
\end{pmatrix} 
\quad \text{and} \quad \nu^*(T) = \mu^*(T) \times \delta_0, 
\end{equation*}
for $\Lcal^1$-almost every $t \in [0,T]$, all $u \in U$ and $\nu^*(t)$-almost every $(x,r) \in \R^d$, and it can then be checked that the curve $\nu^*(\cdot)$ defined above trivially satisfies the statements of Theorem \ref{thm:PMPMayer}. Similarly if $\nabla \varphi(\mu^*(T)) = 0$, we can set $\lambda_0 = 1$ and $\lambda_i =0$ for all $i \in \{1,\dots,n\}$, so that the PMP is again satisfied with $\nu^*(t) = \mu^*(t) \times \delta_0$. Hence, there only remains to consider the case in which $\nabla \varphi(\mu^*(T)) \neq 0$ and $\nabla \Psi_i(\mu^*(T)) \neq 0$ for all $i \in I^{\circ}(\mu^*(T))$. 

If $I^{\circ}(\mu^*(T)) \neq \emptyset$, we introduce the non-empty subset of $\R^k$ defined by
\begin{equation}
\label{eq:BT_Def}
\Bpazo_T := \bigg\{ \Big( \langle \nabla \Psi_1(\mu^*(T)),\G_T \rangle_{L^2(\mu^*(T))},\dots,\langle \nabla \Psi_k(\mu^*(T)),\G_T \rangle_{L^2(\mu^*(T))} \Big) ~\text{s.t.}~ \G_T \in \Rpazo_T^L \bigg\},
\end{equation}
which allows to discriminate between the two scenarios corresponding to the abnormal and normal versions of the maximum principle.


\paragraph*{Step 2: Separation without constraint qualification.} In this case, we suppose that
\begin{equation*}
\Bpazo_T \cap (\R_-^*)^k \, := \, \Bpazo_T \cap (-\infty,0)^k = \emptyset.
\end{equation*}
Notice that $\Bpazo_T$ is a convex subset of $\R^k$ since $\Rpazo_T^L \subset C^0(\R^d,\R^d)$ is convex. Because $(\R_-^*)^k$ is convex as well, there exists by the separation theorem a non-trivial element $p \in \R^k$ such that
\begin{equation}
\label{eq:Abnormal_Sep}
\sup_{d \in (\R_-^*)^k} \big\langle p , d \, \big\rangle \leq \inf_{b \in \Bpazo_T} \big\langle p , b \big\rangle.
\end{equation}
Remark now that since $(\R_-^*)^k$ is a cone, the separation inequality \eqref{eq:Abnormal_Sep} necessarily implies
\begin{equation*}
\sup_{d \in \R_-^k} \langle p , d \, \rangle \leq 0,
\end{equation*}
which yields that $p := (\lambda_1,\dots,\lambda_k) \in \R_+^k$. Similarly, one can check that $\Bpazo_T$ is a cone upon remarking that $\Rpazo_L^T \subset C^0(\R^d,\R^d)$ is itself a cone. Thus, \eqref{eq:Abnormal_Sep} also implies
\begin{equation*}
\inf_{b \in \Bpazo_T} \big\langle p , b \big\rangle \geq 0,
\end{equation*}
which combined with the definition \eqref{eq:BT_Def} of $\Bpazo_T$ further yields
\begin{equation*}
\sum_{i=1}^k \lambda_i \big\langle \nabla \Psi_i(\mu^*(T)),\G_T \big \rangle_{L^2(\mu^*(T))} ~\geq~ 0,
\end{equation*}
for every element $\G_T \in \Rpazo_T^L$. This inequality can be in turn rewritten as
\begin{equation}
\label{eq:Abnormal_Ineq}
\big\langle \hspace{-0.1cm} -P_T , \F \big( T , \Phi^{u^*}_{(T,0)}(\cdot) \big) \big\rangle_{L^2(\mu^*(T))} \leq 0,
\end{equation}
for every solution $\F \in C^0([0,T] \times K,\R^d)$ of \eqref{eq:Linearised_CauchyNonLocal} with  $\F^0 = 0$ and $w(\cdot)$ satisfying \eqref{eq:w_def} and \ref{hyp:CE}, where the covector $P_T \in C^0(K,\R^d)$ is defined by
\begin{equation}
\label{eq:Abnormal_pT}
P_T := \sum_{i=1}^k \lambda_i \nabla \Psi_i(\mu^*(T)).
\end{equation}


\paragraph*{Step 3: Separation with constraint qualification.} We now investigate the scenario in which
\begin{equation*}
\Bpazo_T \cap (\R_-^*)^k \neq \emptyset.
\end{equation*}
In this context, consider the set $\A_T \subset \R^{k+1}$ defined by
\begin{equation*}
\begin{aligned}
\A_T := \bigg\{ \Big( \langle \nabla \varphi(\mu^*(T)) , \G_T \big\rangle_{L^2(\mu^*(T))}, & \langle \nabla \Psi_1(\mu^*(T)),\G_T \rangle_{L^2(\mu^*(T))}, \dots, \\
& \langle \nabla \Psi_k(\mu^*(T)),\G_T \rangle_{L^2(\mu^*(T))} \Big) ~\text{s.t.}~ \G_T \in \Rpazo_T^L \bigg\}.
\end{aligned}
\end{equation*}
In the following lemma, we prove that the set $\A_T$ defined above is disjoint from $(\R_-^*)^{k+1}$.

\begin{lem}[Incompatible intersection]
\label{lem:Intersection}
Let $(\mu^*(\cdot),u^*(\cdot))$ be a strong local minimiser for $(\Ppazo_{\MC})$. Then, it necessarily holds that
\begin{equation}
\label{eq:EmptyInter}
\A_T \cap (\R_-^*)^{k+1} = \emptyset.
\end{equation}
\end{lem}

\begin{proof}
Suppose by contradiction that the intersection in \eqref{eq:EmptyInter} is non-empty. Then, there exists $w(\cdot)$ satisfying \eqref{eq:w_def} and \ref{hyp:CE} such that the solution $\F \in C^0([0,T] \times K,\R^d)$ of \eqref{eq:Linearised_CauchyNonLocal} with $\F^0 = 0$ verifies
\begin{equation}
\label{eq:Contradiction_CostIneq}
\big\langle \nabla \varphi(\mu^*(T)) , \F \big( T , \Phi^{u^*}_{(T,0)}(\cdot) \big) \big\rangle_{L^2(\mu^*(T))} < 0,
\end{equation}
and
\begin{equation}
\label{eq:Contradiction_ConstraintsIneq}
\big\langle \nabla \Psi_i(\mu^*(T)) , \F \big( T , \Phi^{u^*}_{(T,0)}(\cdot) \big) \big\rangle_{L^2(\mu^*(T))} < 0,
\end{equation}
for every $i \in \{ 1,\dots,k\}$. By Step 1, there exists a compact set $K' \subset \R^d$ such that for every $\epsilon > 0$ sufficiently small, there exists a curve $\tilde{\mu}_{\epsilon}(\cdot) \in \AC([0,T],\Pcal(K'))$ solution of \eqref{eq:DiffBis} which satisfies 
\begin{equation*}
\sup_{t \in [0,T]} W_1 \Big( \tilde{\mu}_{\epsilon}(t) , \big( \Id + \epsilon \F \big( t , \Phi^{u^*}_{(t,0)}(\cdot) \big) + o_t(\epsilon) \big)_{\#} \mu^*(t) \Big) = o(\epsilon),
\end{equation*}
where $\sup_{t \in [0,T]} \NormC{o_t(\epsilon)}{0}{K',\R^d} = o(\epsilon)$. Since the solution set of \eqref{eq:DiffBis} coincides with that of \eqref{eq:NonLocalBis}, there exists $\tilde{u}_{\epsilon}(\cdot) \in \U$ such that $(\tilde{\mu}_{\epsilon}(\cdot),\tilde{u}_{\epsilon}(\cdot))$ is a trajectory-control pair for \eqref{eq:NonLocalBis}. Moreover by hypotheses \ref{hyp:MCP}-$(ii)$, the maps $\varphi(\cdot)$ and $\{\Psi_i(\cdot)\}_{1 \leq i \leq n}$ are locally differentiable in the sense of Definition \ref{def:LocalDiff}. Therefore, it holds as a consequence of Corollary \ref{cor:ChainruleBis} together with \eqref{eq:Contradiction_CostIneq} that
\begin{equation}
\label{eq:TaylorMeasureIneq1}
\begin{aligned}
\varphi(\tilde{\mu}_{\epsilon}(T)) & = \varphi \Big( \big( \Id + \epsilon \F\big( T , \Phi_{(T,0)}^{u^*}(\cdot) \big)  + o_T(\epsilon) \big)_{\#} \mu^*(T) \Big) + o(\epsilon)  \\
& = \varphi(\mu^*(T)) + \epsilon \big\langle \nabla \varphi(\mu^*(T)) , \F \big( T , \Phi_{(T,0)}^{u^*}(\cdot) \big) \big\rangle_{L^2(\mu^*(T))} + o_T(\epsilon) \\
& < \varphi(\mu^*(T)),
\end{aligned}
\end{equation}
and analogously, by \eqref{eq:Contradiction_ConstraintsIneq},
\begin{equation}\\
\label{eq:TaylorMeasureIneq2}
\Psi_i(\tilde{\mu}_{\epsilon}(T)) < \Psi_i(\mu^*(T)),
\end{equation}
for every index $i \in \{ 1,\dots, k\}$ and any $\epsilon >0$ small enough, where we also used that fact that the maps $\varphi(\cdot)$ and $\{\Psi_i(\cdot)\}_{i=1}^n$ are Lipschitz continuous over $\Pcal(K')$ in the $W_1$-metric. Combining \eqref{eq:TaylorMeasureIneq1} and \eqref{eq:TaylorMeasureIneq2} while observing that $\Psi_i(\tilde{\mu}_{\epsilon}(T)) < 0$ as well for $i \notin I^{\circ}(\mu^*(T))$ and $\epsilon > 0$ small enough, we have thus built an admissible pair $(\tilde{\mu}_{\epsilon}(\cdot),\tilde{u}_{\epsilon}(\cdot))$ for $(\Ppazo_{\MC})$ that produces a cost strictly smaller than $(\mu^*(\cdot),u^*(\cdot))$, and which satisfies $\sup_{t \in [0,T]} W_1(\mu^*(t),\tilde{\mu}_{\epsilon}(t)) \leq C \epsilon$ for some constant $C > 0$ independent of $\epsilon > 0$. This contradicts our assumption that the pair $(\mu^*(\cdot),u^*(\cdot))$ is a strong local minimiser for $(\Ppazo_{\MC})$.
\end{proof}

As a consequence of Lemma \ref{lem:Intersection}, the intersection described in \eqref{eq:EmptyInter} is necessarily empty. Since both sets in the latter are convex, the separation theorem yields the existence of a non-trivial element $p := (\lambda_0,\dots,\lambda_k) \in \R^{k+1}$ such that
\begin{equation}
\label{eq:SeparationSystem}
\sup_{d \in (\R_-^*)^{k+1}} \langle p , d \, \rangle ~\leq~ \inf_{a \in \A_T} \langle p , a \rangle.
\end{equation}
By repeating the same arguments as in Step 2 above, we again have that $(\lambda_0,\dots,\lambda_k) \in \R_+^{k+1}$. In this case however, it necessarily holds that $\lambda_0 \neq 0$. Indeed if $\lambda_0 = 0$, then \eqref{eq:SeparationSystem} becomes equivalent to $\Bpazo_T \cap (\R_-^*)^k = \emptyset$ because $(\R_-^*)^k$ is open, which contradicts our standing assumption. Hence up to renormalising all the multipliers by $\lambda_0 > 0$, we recover
\begin{equation*}
\big\langle \nabla \varphi(\mu^*(T)), \G_T \big\rangle_{L^2(\mu^*(T))} + \sum_{i=1}^k \lambda_i \big\langle \nabla \Psi_i(\mu^*(T)),\G_T \big \rangle_{L^2(\mu^*(T))} \geq 0,
\end{equation*}
for every $\G_T \in \Rpazo_T^L$. The latter expression can be equivalently rewritten as
\begin{equation*}
\label{eq:Normal_Ineq}
\big\langle \hspace{-0.1cm} -P_T , \F \big( T , \Phi^{u^*}_{(T,0)}(\cdot) \big) \big\rangle_{L^2(\mu^*(T))} \leq 0,
\end{equation*}
where $P_T \in C^0(K,\R^d)$ is given in this context by
\begin{equation}
\label{eq:Normal_pT}
P_T := \nabla \varphi(\mu^*(T)) + \sum_{i=1}^k \lambda_i \nabla \Psi_i(\mu^*(T)).
\end{equation}
Observe now that if $I^{\circ}(\mu^*(T)) = \emptyset$, we can repeat the same arguments to obtain 
\begin{equation*}
\big\langle \nabla \varphi(\mu^*(T)), \G_T \big\rangle_{L^2(\mu^*(T))} \geq 0,
\end{equation*}
for any $\G_T \in \Rpazo_T^L$, and set $P_T := \nabla \varphi(\mu^*(T))$. 


\paragraph*{Step 4: Proof of the PMP.}

Condensing the results of Step 2 and Step 3 and setting $\lambda_i = 0$ for all $i \in \{k+1,\dots,n\}$, we have built a covector $P_T \in C^0(K,\R^d)$ given explicitly by
\begin{equation}
\label{eq:EndPointMultiplier}
P_T = \lambda_0 \nabla \varphi(\mu^*(T)) + \sum_{i=1}^n \lambda_i \nabla  \Psi_i(\mu^*(T)),
\end{equation}
for some $(\lambda_0,\lambda_1,\dots,\lambda_n) \in \{0,1\} \times \R_+^n$ not all equal to $0$ and satisfying $\lambda_i \Psi_i(\mu^*(T)) = 0$ for every $i \in \{ 1,\dots,n\}$. Moreover, the covector $P_T$ is such that the family of end-point inequalities
\begin{equation}
\label{eq:EndPointIneq1}
\INTDom{ \big\langle  \hspace{-0.1cm} -P_T(x) , \F \big( T , \Phi_{(T,0)}^{u^*}(x) \big) \big\rangle}{\R^d}{\mu^*(T)(x)} \leq 0,
\end{equation}
hold for any solution $\F \in C^0([0,T] \times K,\R^d)$ of \eqref{eq:Linearised_CauchyNonLocal} with $\F^0 = 0$ and $w(\cdot)$ satisfying \eqref{eq:w_def} and \ref{hyp:CE}. Our goal now is to build a curve $\nu^*(\cdot)$ solution of the \textit{forward-backward} continuity equation
\begin{equation}
\label{eq:ForwardBackward_NonLocal}
\left\{
\begin{aligned}
& \partial_t \nu^*(t) + \Div \Big( \V(t,\nu^*(t)) \nu^*(t) \Big) = 0, \\
& \pi^1_{\#} \nu^*(t) = \mu^*(t) \hspace{0.55cm} \text{for all times $t \in [0,T]$}, \\
& \nu^*(T) = (\Id , -P_T)_{\#} \mu^*(T),
\end{aligned}
\right.
\end{equation}
where the non-local velocity field $\V : [0,T] \times \Pcal_c(\R^{2d}) \times \R^{2d} \rightarrow \R^{2d}$ is defined by
\begin{equation}
\label{eq:PMP_NonLocalDef}
\V(t,\nu,x,r) := \begin{pmatrix}
v \big( t , \pi^1_{\#} \nu ,u^*(t), x \big) \\ \\ -\D_x v \big( t , \pi^1_{\#} \nu,u^*(t),x \big)^{\hspace{-0.1cm}\top} r -  \mathlarger{\INTDom{\D_{\mu} v \big(t,\pi^1_{\#} \nu,u^*(t),y \big)(x)^{\top} p \,  }{\R^{2d}}{\nu(y,p)}}
\end{pmatrix},
\end{equation}
for $\Lcal^1$-almost every $t \in [0,T]$ and any $(\nu,x,r) \in \Pcal_c(\R^{2d}) \times \R^{2d}$.

Observe that $\V(\cdot,\cdot,\cdot)$ does not satisfy hypotheses \ref{hyp:CE}, so that we cannot directly apply the existence result of Corollary \ref{cor:NonLocalPDE} to assert that \eqref{eq:ForwardBackward_NonLocal} admits solutions. The following lemma, originally established in \cite{PMPWass}, provides an explicit disintegration construction of such solutions by exploiting the \textit{cascaded} structure of the system (see also \cite{PMPWassConst}).

\begin{lem}[Definition and existence of solutions to  \eqref{eq:ForwardBackward_NonLocal}]
\label{lem:HamiltonianConstruction}
Let $(\mu^*(\cdot),u^*(\cdot))$ be a strong local minimiser for $(\Ppazo_{\MC})$ and suppose that hypotheses \textnormal{\ref{hyp:MCP}} hold. For $\mu^*(T)$-almost every $x \in \R^d$, denote by $(\Psi^x_{(T,t)}(\cdot))_{t \in [0,T]}$ the family of \textnormal{backward non-local flows}, solutions of
\begin{equation*}
\left\{
\begin{aligned}
\partial_t \Psi^x_{(T,t)}(r) & = - \, \D_x v \Big( t, \mu^*(t) ,u^*(t) , \Phi^{u^*}_{(T,t)}(x)\Big)^{\hspace{-0.1cm} \top} \Psi_{(T,t)}^x(r) \\
& \hspace{0.45cm} -\INTDom{\bigg( \D_{\mu} v \Big( t,\mu^*(t),u^*(t) , \Phi_{(T,t)}^{u^*}(y) \Big) \big( \Phi^{u^*}_{(T,t)}(x) \big)^{\hspace{-0.1cm} \top} \Psi^y_{(T,t)}(p) \bigg)}{\R^{2d}}{\big( (\Id,-P_T)_{\#} \mu^*(T) \big)(y,p)}, \\
\Psi^x_{(T,T)}(r) & = r,
\end{aligned}
\right.
\end{equation*}
and define the curve of measures $\sigma^*_x(\cdot) \in \AC([0,T],\Pcal_c(\R^d))$ as
\begin{equation*}
\sigma^*_x(t) := \Psi^x_{(T,t)}(\cdot)_{\#} \delta_{(-P_T(x))},
\end{equation*}
for all times $t \in [0,T]$. Then, the curve given by $\nu^* : t \in [0,T] \mapsto \big( \Phi^{u^*}_{(T,t)} \circ \pi^1, \pi^2 \big)_{\#} \nu_T^*(t)$ with
\begin{equation*}
\nu^*_T(t) := \INTDom{\sigma_x^*(t)}{\R^d}{\mu^*(T)(x)},
\end{equation*}
is a solution of \eqref{eq:ForwardBackward_NonLocal}. Moreover, there exist $R_r'>0$ and $m_r'(\cdot) \in L^1([0,T],\R_+)$ such that
\begin{equation*}
\supp(\nu^*(t)) \subset K' \times K' \qquad \text{and} \qquad W_1(\nu^*(t),\nu^*(s)) \leq \INTSeg{m_r'(\tau)}{\tau}{s}{t},
\end{equation*}
for all times $0 \leq s \leq t \leq T$, where $K' := B(0,R_r')$.
\end{lem}

\begin{proof}
See \cite[Lemma 3.7]{PMPWassConst} and \cite[Lemma 6]{PMPWass}.
\end{proof}

Applying to the Hamiltonian $\nu \in \Pcal_c(\R^d) \mapsto \H(t,\nu,u) \in \R$ defined in \eqref{eq:Thm_HamiltonianDef} the results of Section \ref{appendix:Examples}, which provide analytical expressions for the gradients of smooth integral functionals, one can check
\begin{equation*}
\V(t,\nu^*(t),x,r) = \J_{2d} \nabla_{\nu} \H(t,\nu^*(t),u^*(t))(x,r),
\end{equation*}
for $\Lcal^1$-almost every $t \in [0,T]$ and any $(x,r) \in \R^{2d}$. Since $\nu^*(T) = (\Id,-P_T)_{\#} \mu^*(T)$ by the construction detailed in Lemma \ref{lem:HamiltonianConstruction} with $P_T$ being given by \eqref{eq:EndPointMultiplier}, the curve $\nu^*(\cdot)$ is a solution of \eqref{eq:Thm_HamiltonianSys}. In the next lemma, we state an auxiliary result which will in turn yield the maximisation condition.

\begin{lem}[Derivative of an auxiliary functional]
\label{lem:AuxQuant}
Given $w(\cdot)$ satisfying \eqref{eq:w_def} and \ref{hyp:CE}, define
\begin{equation}
\label{eq:H_Def}
\Hpazo : t \in [0,T] \mapsto \INTDom{\big\langle r , \F \big(t, \Phi_{(t,0)}^{u^*}(x) \big) \big\rangle}{\R^{2d}}{\nu^*(t)(x,r)},
\end{equation}
where $\F \in C^0([0,T] \times K,\R^d)$ is the corresponding solution of \eqref{eq:Linearised_CauchyNonLocal} with $\F^0 = 0$. Then, $\Hpazo(\cdot) \in \AC([0,T],\R)$ and its pointwise derivative is given explicitly by
\begin{equation}
\label{eq:H_Derivative}
\tderv{}{t} \Hpazo(t) = \INTDom{\langle r , w(t,x) \rangle}{\R^{2d}}{\nu^*(t)(x,r)},
\end{equation}
for $\Lcal^1$-almost every $t \in [0,T]$.
\end{lem}

\begin{proof}
The proof of this result is a matter of fairly long computations which mimic those already performed in \cite[Lemma 3.9]{PMPWassConst} and \cite[Lemma 7]{PMPWass}.
\end{proof}

Observe that by the definition of $\nu^*(\cdot)$ given in Lemma \ref{lem:HamiltonianConstruction} together with \eqref{eq:EndPointIneq1}, one has
\begin{equation*}
\INTDom{\big\langle r , \F \big(T, \Phi_{(T,0)}^{u^*}(x) \big) \big\rangle}{\R^{2d}}{\nu^*(T)(x,r)} \leq 0,
\end{equation*}
which can in turn be reformulated as
\begin{equation}
\label{eq:H_Ineq}
\Hpazo(T) \leq 0,
\end{equation}
with $\Hpazo(\cdot)$ being given as in \eqref{eq:H_Def}. Remark now that since $\Hpazo(\cdot) \in \AC([0,T],\R)$ and $\F(0) = \F^0 =  0$ by construction, it further holds by Lemma \ref{lem:AuxQuant}
\begin{equation}
\label{eq:H_Integral}
\Hpazo(T) = \INTSeg{\tderv{}{t} \Hpazo(t)}{t}{0}{T} = \INTSeg{\INTDom{\langle r , w(t,x) \rangle}{\R^{2d}}{\nu^*(t)(x,r)}}{t}{0}{T},
\end{equation}
so that upon merging \eqref{eq:H_Ineq} and \eqref{eq:H_Integral}, we obtain
\begin{equation}
\label{eq:Tangent_IntegralIneq}
\INTSeg{\INTDom{\langle r , w(t,x) \rangle}{\R^{2d}}{\nu^*(t)(x,r)}}{t}{0}{T} \leq 0,
\end{equation}
for any $w(\cdot)$ satisfying \eqref{eq:w_def} and \ref{hyp:CE}.

Let us consider for some given $t \in [0,T]$ and $m \geq 1$ the closed subset of controls
\begin{equation*}
\tilde{U}_m(t) := \Big\{ u \in U ~\text{s.t.}~ \H(t,\nu^*(t),u) \geq \H(t,\nu^*(t),u^*(t)) + \tfrac{1}{m} \Big\},
\end{equation*}
and suppose that the associated $\Lcal^1$-measurable subset $\Omega_m \subset [0,T]$, defined by
\begin{equation*}
\Omega_m := \big\{ t \in [0,T] ~\text{s.t.}~ \tilde{U}_m(t) \neq \emptyset  \big\},
\end{equation*}
has positive measure for some $m \geq 1$. Then by Theorem \cite[Theorem 8.2.9]{Aubin1990}, we can find a measurable selection $t \in [0,T]\mapsto \tilde{u}(t) \in U$ such that $\tilde{u}(t) \in \tilde{U}_m(t)$ for $\Lcal^1$-almost every $t \in \Omega_m$ and $\tilde{u}(t) =  u^*(t)$ otherwise. Observe next that the map defined by
\begin{equation*}
\tilde{w}(t) := v(t,\mu^*(t),\tilde{u}(t)) - v(t,\mu^*(t),u^*(t)),
\end{equation*}
for all times $t \in [0,T]$ satisfies \eqref{eq:w_def} and hypotheses \ref{hyp:CE}. Moreover, it is such that
\begin{equation*}
\INTSeg{\INTDom{\langle r , \tilde{w}(t,x) \rangle}{\R^{2d}}{\nu^*(t)(x,r)}}{t}{0}{T} =\INTDom{\Big( \H(t,\nu^*(t),\tilde{u}(t)) - \H(t,\nu^*(t),u^*(t)) \Big)}{\Omega_m}{t} \geq \tfrac{1}{m} \Lcal^1(\Omega_m),
\end{equation*}
which violates \eqref{eq:Tangent_IntegralIneq}. Thus, the set $\Omega := \cup_{m \geq 1} \Omega_m \subset [0,T]$ necessarily has zero measure, which by definition of the sets $\tilde{U}_m(t)$ for $\Lcal^1$-almost every $t \in [0,T]$ yields the maximisation condition \eqref{eq:Thm_Maximisation}.

\begin{rmk}[On the choice of performing separations on $\A_T$ and $\Bpazo_T$]
Usually, geometric proofs of the PMP and its second-order variants tend to involve separation arguments on sets of trajectories which are not transposable to the setting of Wasserstein spaces. Thus in Step 2 and Step 3, we chose to perform separation arguments directly on $\A_T$ and $\Bpazo_T$, which are subsets of finite-dimensional euclidean spaces defined as the images of the reachable set $\Rpazo_T^L$ under the action of the gradients of the cost and constraint functionals. As illustrated before, this choice allows for a very simple and concise discrimination between the abnormal and normal scenarios of the PMP.
\end{rmk}


\subsection{Adaptation of the PMP to constrained Bolza problems}
\label{subsection:Bolza}

In this section, we make use of the results of Section \ref{subsection:Mayer} to obtain an extension of Theorem \ref{thm:PMPMayer} to the setting of general constrained Bolza problems of the form
\begin{equation*}
(\Ppazo_\BC) ~ \left\{
\begin{aligned}
\min_{u(\cdot) \in \U} & \left[ \INTSeg{L(t,\mu(t),u(t))}{t}{0}{T} + \varphi(\mu(T)) \right] \\
\text{s.t.}~ & \left\{
\begin{aligned}
& \partial_t \mu(t) + \Div \Big( v (t,\mu(t),u(t)) \mu(t) \Big) = 0, \\
& \mu(0) \, = \mu^0, \\
& \mu(T) \in \Qpazo_T, \\
\end{aligned}
\right.
\end{aligned}
\right.
\end{equation*}
where $L:[0,T] \times \mathcal{P}_c(\mathbb{R}^d) \times U \mapsto \mathbb{R}$ is a given running cost which satisfies the following.

\begin{taggedhyp}{\textbn{(L)}}
\label{hyp:L}
For every $R > 0$, assume that the following holds with $K := B(0,R)$.
\begin{enumerate}
\item[$(i)$] The map $t \in [0,T] \mapsto L(t,\mu,u)$ is $\Lcal^1$-measurable for any $(\mu,u) \in \Pcal_c(\R^d) \times U$. Moreover, there exists $k(\cdot) \in L^1([0,T],\R_+)$ such that $\sup_{u \in U}|L(t,\delta_0,u)| \leq k(t)$ for $\Lcal^1$-almost every $t \in [0,T]$.
\item[$(ii)$] The map $u \in U \mapsto L(t,\mu,u)$ is continuous for $\Lcal^1$-almost every $t \in [0,T]$ and any $\mu \in \Pcal_c(\R^d)$.
\item[$(iii)$] The map $\mu \in \Pcal_c(\R^d) \mapsto L(t,\mu,u)$ is locally differentiable in the sense of Definition \ref{def:LocalDiff} and the application $x \in \R^d \mapsto \nabla_{\mu} L(t,\mu,u)(x) \in \R^d$ is continuous for $\Lcal^1$-almost every $t \in [0,T]$ and any $(\mu,u) \in \Pcal_c(\R^d) \times U$. Moreover, there exists a map $\Lpazo_K(\cdot) \in L^1([0,T],\R_+)$ such that
\begin{equation*}
|L(t,\mu,u) - L(t,\nu,u)| \leq \Lpazo_K(t) W_1(\mu,\nu),
\end{equation*}
for every $\mu,\nu \in \Pcal(K)$ and any $u \in U$.
\end{enumerate}
\end{taggedhyp}

In the present context, we say that an admissible pair $(\mu^*(\cdot),u^*(\cdot)) \in \AC([0,T],\Pcal_c(\R^d)) \times \U$ is a strong local minimiser for $(\Ppazo_{\BC})$ if there exists $\epsilon > 0$ such that
\begin{equation*}
\INTSeg{L(t,\mu^*(t),u^*(t))}{t}{0}{T} + \varphi(\mu^*(T)) \leq \INTSeg{L(t,\mu(t),u(t))}{t}{0}{T} + \varphi(\mu(T)) ,
\end{equation*}
for every admissible pair $(\mu(\cdot),u(\cdot))$ which satisfies $\sup_{t \in [0,T]} W_1(\mu^*(t),\mu(t))  \leq \epsilon$.

\begin{thm}[Pontryagin Maximum Principe for $(\Ppazo_{\BC})$]
\label{thm:PMPBolza}
Let $(\mu^*(\cdot),u^*(\cdot)) \in \AC([0,T],\Pcal_c(\R^d)) \times \U$ be a strong local minimiser for $(\Ppazo_{\BC})$, and suppose that hypotheses \textnormal{\ref{hyp:MCP}} and \textnormal{\ref{hyp:L}} hold.

Then, the conclusions of Theorem \ref{thm:PMPMayer} hold with the Hamiltonian $\H_{\lambda_0} : [0,T] \times \Pcal_c(\R^{2d}) \times U \rightarrow \R$ associated to $(\Ppazo_{\BC})$, defined by
\begin{equation}
\label{eq:Thm_BolzaHamiltonian}
\H_{\lambda_0}(t,\nu,u) := \INTDom{\big\langle r , v(t,\pi^1_{\#} \nu,u,x) \big\rangle}{\R^{2d}}{\nu(x,r)} - \lambda_0 L(t,\pi^1_{\#}\nu,u),
\end{equation}
for all $(t,\nu,u) \in [0,T] \times \Pcal_c(\R^{2d}) \times U$, which Wasserstein gradient writes
\begin{equation}
\label{eq:Thm_BolzaGrad}
\begin{aligned}
& \nabla_{\nu} \H_{\lambda_0}(t,\nu^*(t),u^*(t))(x,r) \\
& \hspace{1cm} = \begin{pmatrix}
 \D_x v \big( t,\mu^*(t),u^*(t),x \big)^{\top} r + \INTDom{\D_{\mu} v \big( t,\mu^*(t),u^*(t),y \big)(x)^{\top} p \,}{\R^{2d}}{\nu^*(t)(y,p)} \\ \hspace{1.05cm} - \lambda_0 \nabla_{\mu} L(t,\mu^*(t),u^*(t))(x) \\ \\
v \big( t,\mu^*(t),u^*(t),x \big)
\end{pmatrix},
\end{aligned}
\end{equation}
for $\Lcal^1$-almost every $t \in [0,T]$ and any $(x,r) \in \R^{2d}$.
\end{thm}

In what follows, we show how the statements of Theorem \ref{thm:PMPBolza} can be recovered by directly applying Theorem \ref{thm:PMPMayer} to a suitable augmented Mayer problem. We shall henceforth denote by $\Bpi^1 : \R^{d+1} \times \R^{d+1} \rightarrow\R^{d+1}$ and $\Bpi^2 : \R^{d+1} \times \R^{d+1} \rightarrow\R^{d+1}$ the projection operators onto the first and second factor.

Given an admissible pair $(\mu(\cdot),u(\cdot))$ for $(\Ppazo_{\BC})$, let $z(\cdot) \in \AC([0,T],\R)$ be the curve defined as 
\begin{equation}
\label{eq:zdef}
z(t) := \INTSeg{L(s,\mu(s),u(s))}{s}{0}{t},
\end{equation}
and consider the corresponding \textit{augmented state}
\begin{equation*}
\Bmu(t) := \mu(t) \times \delta_{z(t)},
\end{equation*}
associated to $(\Ppazo_{\BC})$, both defined for all times $t \in [0,T]$. Below given a measure $\Bmu \in \Pcal_c(\R^{d+1})$, we will set $\mu := (\pi^1,\dots,\pi^d)_{\#} \Bmu \in \Pcal_c(\R^d)$. It is then clear that if $(\mu^*(\cdot),u^*(\cdot))$ is a strong local minimiser for $(\Ppazo_{\BC})$, then choosing $z^*(\cdot)$ as in \eqref{eq:zdef} with $(\mu(\cdot),u(\cdot)) := (\mu^*(\cdot),u^*(\cdot))$ and defining $\Bmu^*(t) := \mu^*(t) \times \delta_{z^*(t)}$, the pair $(\Bmu^*(\cdot),u^*(\cdot))$ is a strong local minimiser for the Mayer problem
\begin{equation*}
(\hat{\Ppazo}_{\MC}) ~ \left\{
\begin{aligned}
\min_{u(\cdot) \in \U} & \left[ \hat{\varphi}(\Bmu(T)) \right], \\
\text{s.t.}~ \, & \left\{
\begin{aligned}
& \partial_t \Bmu(t) + \Div \Big( \hat{v} (t,\Bmu(t),u(t)) \Bmu(t) \Big) = 0, \\
& \Bmu(0) \, = \mu^0 \times \delta_0, \\
& \Bmu(T) \in \hat{\Qpazo}_T, \\
\end{aligned}
\right.
\end{aligned}
\right.
\end{equation*}
where we defined
\begin{equation}
\label{eq:Extended_Functionals}
\hat{\varphi}(\Bmu) := \varphi \big(\mu \big) + \INTDom{\hspace{-0.15cm} z \,}{\R^{d+1}}{\Bmu(x,z)} \quad \text{and} \quad \hat{v}(t,\Bmu,u,x,z) := \begin{pmatrix}
v(t,\mu,u,x) \, \\ L(t,\mu,u)
\end{pmatrix},
\end{equation}
for all $(t,\Bmu,u,x,z) \in [0,T] \times \Pcal_c(\R^{d+1}) \times U \times \R^{d+1}$, along with
\begin{equation*}
\hat{\Qpazo}_T := \Big\{ \Bmu \in \Pcal_2(\R^{d+1}) ~\text{s.t.}~ \hat{\Psi}_i(\Bmu) \leq 0 ~~ \text{for every $i \in \{1,\dots,n\}$} \Big\},
\end{equation*}
where $\hat{\Psi}_i(\Bmu) := \Psi_i(\mu)$. By applying the results of Section \ref{appendix:Examples}, the first-order derivatives of the augmented dynamics, final cost and constraint maps can be written explicitly as
\begin{equation}
\label{eq:Extended_Derivatives1}
\D_{(x,z)} \hat{v} \big( t , \Bmu , u , x ,z \big) =
\begin{pmatrix}
\D_x v \big( t,\mu,u,x \big) \\ 0
\end{pmatrix}, \qquad
\D_{\Bmu} \hat{v}(t,\Bmu,u,x,z) =
\begin{pmatrix}
\, \D_{\mu} v(t,\mu,u,x) & 0 \, \\ \, \nabla_{\mu} L(t,\mu,u) & 0 \,
\end{pmatrix},
\end{equation}
and
\begin{equation}
\label{eq:Extended_Derivatives2}
\nabla \hat{\varphi}(\Bmu) = \begin{pmatrix}
\nabla \varphi \big( \mu \big) \\ 1
\end{pmatrix}, \qquad
\nabla \hat{\Psi}_i(\Bmu) = \begin{pmatrix}
\nabla \Psi_i \big( \mu \big) \\ 0
\end{pmatrix},
\end{equation}
for $\Lcal^1$-almost every $t \in [0,T]$, all $(\Bmu,u,x,z) \in \Pcal_c(\R^{d+1}) \times U \times \R^{d+1}$ and any index $i \in \{1,\dots,n\}$.

By Corollary \ref{cor:NonLocalPDE}, there exists $R_r > 0$ such that $\supp(\mu^*(t)) \subset K := B(0,R_r)$ for all times $t \in [0,T]$ where $r > 0$ is such that $\mu^0 \in \Pcal(B(0,r))$. Hence, observe that under hypotheses \ref{hyp:L}-$(iii)$ and up to redefining the running cost for $\Lcal^1$-almost every $t \in [0,T]$ and any $u \in U$ as
\begin{equation*}
L(t,\mu,u) := L(t,(\pi_K)_{\#} \mu,u),
\end{equation*}
whenever $\supp(\mu) \not\subseteq K$, the extended velocity field $\hat{v} : [0,T] \times \Pcal(\R^{d+1}) \times U \times \R^{d+1} \mapsto \R^{d+1}$ satisfies hypotheses \ref{hyp:H} with $p=1$. This together with the definitions of $\hat{\varphi}(\cdot)$ and $\{ \hat{\Psi}_i(\cdot)\}_{1 \leq i \leq n}$ implies that hypotheses \ref{hyp:MCP} hold for $(\hat{\Ppazo}_{\MC})$. Thus by Theorem \ref{thm:PMPMayer}, there exist non-trivial multipliers $(\lambda_0,\dots,\lambda_n) \in \{0,1\} \times \R_+^n$ and a curve of measures $\Bnu^*(\cdot) \in \AC([0,T],\Pcal_c((\R^{d+1})^2)$ solution of the forward-backward Hamiltonian continuity equation
\begin{equation}
\label{eq:Extended_Dyn}
\left\{
\begin{aligned}
& \partial_t \Bnu^*(t) + \Div \Big( \J_{2(d+1)} \nabla_{\Bnu} \hat{\H}(t,\Bnu^*(t),u^*(t)) \Bnu^*(t) \Big) = 0, \\
& \pi^1_{\#} \Bnu^*(t) = \mu^*(t) \times \delta_{z^*(t)} \hspace{0.95cm} \text{for all times $t \in [0,T]$}, \\
& \Bnu^*(T) = \bigg( \Id \, , \, - \lambda_0 \nabla \hat{\varphi}(\Bmu^*(T)) - \sum_{i=1}^n \lambda_i \nabla \hat{\Psi}_i(\Bmu^*(T)) \bigg)_{\raisebox{6pt}{$\scriptstyle \#$}} \Bmu^*(T),
\end{aligned}
\right.
\end{equation}
such that the complementarity slackness condition 
\begin{equation}
\label{eq:Extended_Comp}
\lambda_i \hat{\Psi}_i(\Bmu^*(T)) = 0,
\end{equation}
holds for every index $i \in \{1,\dots,n\}$, and the maximisation condition 
\begin{equation}
\label{eq:Extended_Max}
\hat{\H}(t,\Bnu^*(t),u^*(t)) = \max_{u \in U} \, \hat{\H}(t,\Bnu^*(t),u).
\end{equation}
is satisfied for $\Lcal^1$-almost every $t \in [0,T]$. Here, the \textit{augmented Hamiltonian} $\hat{\H} : [0,T] \times \Pcal_c((\R^{d+1})^2) \times U \rightarrow \R$ associated to $(\hat{\Ppazo}_{\MC})$ is defined by
\begin{equation*}
\hat{\H}(t,\Bnu,u) := \INTDom{\big\langle (r,q) , \hat{v}(t,\Bpi^1_{\#} \Bnu,u,x,z) \big\rangle}{\R^{2d+2}}{\Bnu(x,z,r,q)},
\end{equation*}
for all $(t,\Bnu,u) \in [0,T] \times \Pcal_c((\R^{d+1})^2) \times U$. 

Observe now that by \eqref{eq:Extended_Functionals}, the augmented Hamiltonian can be further expressed as 
\begin{equation}
\label{eq:HamiltonianCorresp}
\begin{aligned}
\hat{\H}(t,\Bnu^*(t),u^*(t)) & = \INTDom{\Big\langle (r,q) , \hat{v}(t,\Bpi^1_{\#} \Bnu^*(t),u^*(t),x,z) \Big\rangle}{\R^{2d+2}}{\Bnu^*(t)(x,r,z,q)} \\
& = \INTDom{\Big( \langle r , v(t,\mu^*(t),u^*(t),x) \rangle + q L(t,\mu^*(t),u^*(t)) \Big)}{\R^{2d+2}}{\Bnu^*(t)(x,z,r,q)} 
\end{aligned}
\end{equation}
which along with \eqref{eq:Extended_Derivatives1} allows us to derive the following analytical expression for its gradient 
\begin{equation}
\label{eq:HamiltonianCorrespGrad}
\begin{aligned}
\nabla_{\Bnu} \hat{\H}(t,\Bnu^*(t),u^*(t))(x,z,r,q) & =  \begin{pmatrix}
 \D_{(x,z)} \hat{v} \big( t,\Bmu^*(t),u^*(t),x,z \big)^{\top} (r,q) \\
 + \INTDom{\D_{\Bmu} \hat{v} \big( t,\Bmu^*(t),u^*(t),x',z'\big)(x,z)^{\top} (r',q') \,}{\R^{2d+2}}{\Bnu^*(t)(x',z',r',q')} \\ \\
\hat{v} \big( t,\Bmu^*(t),u^*(t),x \big) 
\end{pmatrix} \\ 
& =  \begin{pmatrix}
\, \D_x v \big( t,\mu^*(t),u^*(t),x \big)^{\top} r + q  \nabla_{\mu} L(t,\mu^*(t),u^*(t))(x) \\ + \INTDom{\D_{\mu} v \big( t,\mu^*(t),u^*(t),y \big)(x)^{\top} p \,}{\R^{2d}}{\nu^*(t)(y,p)} \\ 0 \\ \\ v \big( t,\mu^*(t),u^*(t),x \big) \\ L(t,\mu^*(t),u^*(t))
\end{pmatrix}.
\end{aligned}
\end{equation}
It can now be verified straightforwardly as a consequence of \eqref{eq:Extended_Derivatives2} that
\begin{equation*}
\pi^{2d+2}_{\#} \Bnu^*(T) = \delta_{(-\lambda_0)},
\end{equation*}
and also that the right-hand side of the dynamics driving this extra covector variable is identically equal to $0$ by plugging \eqref{eq:HamiltonianCorrespGrad} into \eqref{eq:Extended_Dyn}. Hence up to a permutation of the coordinates, it is possible to express the extended state-costate curve $\Bnu^*(\cdot)$ for all times $t \in [0,T]$ as
\begin{equation}
\label{eq:Bnu_Decomposition}
\Bnu^*(t) = \nu^*(t) \times \delta_{(z^*(t),-\lambda_0)},
\end{equation}
where $\nu^*(\cdot) \in \AC([0,T],\Pcal_c(\R^{2d}))$. Plugging this last expression into \eqref{eq:HamiltonianCorresp} and \eqref{eq:HamiltonianCorrespGrad} then yields 
\begin{equation}
\label{eq:HamiltonianCorrespBis}
\hat{\H}(t,\Bnu^*(t),u^*(t)) = \H_{\lambda_0}(t,\nu^*(t),u^*(t)), 
\end{equation}
and by repeating the same coordinate permutation, we recover the following expression for the gradient
\begin{equation*}
\begin{aligned}
\nabla_{\Bnu} \hat{\H}(t,\Bnu^*(t),u^*(t))(x,z,r,q) & =  \begin{pmatrix}
\, \D_x v \big( t,\mu^*(t),u^*(t),x \big)^{\top} r  -\lambda_0  \nabla_{\mu} L(t,\mu^*(t),u^*(t))(x) \\ + \INTDom{\D_{\mu} v \big( t,\mu^*(t),u^*(t),y \big)(x)^{\top} p \,}{\R^{2d}}{\nu^*(t)(y,p)} \\ v \big( t,\mu^*(t),u^*(t),x \big) \\ \\ L(t,\mu^*(t),u^*(t)) \\ 0
\end{pmatrix},
\end{aligned}
\end{equation*}
for $\Lcal^1$-almost every $t \in [0,T]$ and any $(x,z,r,q) \in \supp(\Bnu^*(t))$. By applying again the differentiation results of Proposition \ref{prop:Example}, it can finally be checked that 
\begin{equation*}
\nabla_{\nu} \H_{\lambda_0}(t,\nu^*(t),u^*(t))(x,r) = \begin{pmatrix}
\, \D_x v \big( t,\mu^*(t),u^*(t),x \big)^{\top} r  -\lambda_0  \nabla_{\mu} L(t,\mu^*(t),u^*(t))(x) \\ + \INTDom{\D_{\mu} v \big( t,\mu^*(t),u^*(t),y \big)(x)^{\top} p \,}{\R^{2d}}{\nu^*(t)(y,p)} \\ v \big( t,\mu^*(t),u^*(t),x \big)
\end{pmatrix}, 
\end{equation*}
and the gradient of the augmented Hamiltonian can be rewritten as  
\begin{equation}
\label{eq:HamiltonianGradFinal}
\nabla_{\Bnu} \hat{\H}(t,\Bnu^*(t),u^*(t))(x,z,r,q) =  \begin{pmatrix}
\nabla_{\nu} \H_{\lambda_0}(t,\nu^*(t),u^*(t))(x,r) \\ L(t,\mu^*(t),u^*(t)) \\ 0
\end{pmatrix}.
\end{equation}
Thus because $\Bnu^*(\cdot)$ solves \eqref{eq:Extended_Dyn} and owing to the decomposition \eqref{eq:Bnu_Decomposition} combined with the identity \eqref{eq:HamiltonianGradFinal}, it follows that $\nu^*(\cdot)$ is a solution of \eqref{eq:Thm_HamiltonianSys} driven by gradient of the Hamiltonian defined in \eqref{eq:Thm_BolzaHamiltonian}. Moreover, the complementarity slackness and maximisation condition can be directly recovered by using the definition of $\hat{\Psi}_i(\Bmu^*(T))$ in \eqref{eq:Extended_Comp} for the former and plugging \eqref{eq:HamiltonianCorrespBis} into \eqref{eq:Extended_Max} for the latter.


\section{Examples of locally differentiable functionals}
\label{appendix:Examples}

In this auxiliary section, we provide two examples of extended real-valued functionals defined over $\Pcal_2(\R^d)$ which are locally differentiable in the sense of Definition \ref{def:LocalDiff}, and compute their gradients.

\begin{prop}[Locally differentiable integral functionals]
\label{prop:Example}
Let $V : \R^d \rightarrow \R$ be a continuously differentiable mapping. Then the functional
\begin{equation}
\label{eq:Vdef}
\V : \mu \in \Pcal_c(\R^d) \mapsto \INTDom{V(x)}{\R^d}{\mu(x)} \in \R,
\end{equation}
is locally differentiable in the sense of Definition \ref{def:LocalDiff}, and its Wasserstein gradient is given by
\begin{equation*}
\nabla \V (\mu)(x) = \nabla V(x),
\end{equation*}
for every $\mu \in \Pcal_c(\R^d)$ and any $x \in \supp(\mu)$.

Similarly, let $L : \R^d \times \Pcal_c(\R^d) \rightarrow \R$ be such that $x \in \R^d \rightarrow L(x,\mu)$ is continuously differentiable for every $\mu \in \Pcal_c(\R^d)$ and $\mu \in \Pcal_c(\R^d) \mapsto \nabla_x L(x,\mu)$ is continuous for every $x \in \R^d$. Moreover, suppose that $\mu \in \Pcal_c(\R^d) \rightarrow L(x,\mu)$ is locally differentiable for every $x \in \R^d$ and that $(x,y) \in \R^d \times \supp(\mu) \mapsto \nabla_{\mu} L(x,\mu)(y) \in \R^d$ is continuous for every $\mu \in \Pcal_c(\R^d)$. Then, the map
\begin{equation}
\label{eq:Ldef}
\Lpazo : \mu \in \Pcal_c(\R^d) \mapsto \INTDom{L(x,\mu)}{\R^d}{\mu(x)},
\end{equation}
is locally differentiable, and its Wasserstein gradient writes
\begin{equation*}
\nabla \Lpazo(\mu)(x) = \nabla_x L(x,\mu) + \INTDom{\nabla_{\mu} L(z,\mu)(x)}{\R^d}{\mu(z)} ,
\end{equation*}
for any $\mu \in \Pcal_c(\R^d)$ and $x \in \supp(\mu)$.
\end{prop}

\begin{proof}
We first prove that the functional $\V(\cdot)$ defined in \eqref{eq:Vdef} is locally differentiable. Let $\mu \in \Pcal_c(\R^d)$, $R > 0$ and $\nu \in \Pcal(B_{\mu}(R))$. For every $\gamma \in \Gamma_o(\mu,\nu)$, it holds
\begin{equation}
\label{eq:Appendix_TaylorV}
\begin{aligned}
\V(\nu) - \V(\mu) & = \INTDom{ \Big( V(y) - V(x) \Big)}{\R^{2d}}{\gamma(x,y)} \\
& = \INTDom{\INTSeg{\langle \nabla V \big( x + s(y-x) \big) , y-x \rangle}{s}{0}{1}}{\R^{2d}}{\gamma(x,y)}, \\
& = \INTDom{\langle \nabla V(x) , y-x \rangle}{\R^{2d}}{\gamma(x,y)} + \INTSeg{\INTDom{\langle \nabla V \big( x + s(y-x) \big) - \nabla V(x) , y-x \rangle}{\R^{2d}}{\gamma(x,y)}}{s}{0}{1}, \\
& = \INTDom{\langle \nabla V(x) , y-x \rangle}{\R^{2d}}{\gamma(x,y)} + I_R,
\end{aligned}
\end{equation}
where we used Fubini's theorem. Our goal is now to prove that $I_R = o_R(W_2(\mu,\nu))$. Since we assumed that $\nabla V(\cdot)$ is continuous, we can find for every $\epsilon > 0$ another real $\eta > 0$ such that $|\nabla V(x+z) - \nabla V(x)| \leq \tfrac{\epsilon}{2}$ for any $z \in B(0,\eta)$. Whence, we can split $I$ into two terms and estimate them as
\begin{equation*}
\begin{aligned}
|I_R| & \leq \INTSeg{\INTDom{|\langle \nabla V \big( x + s(y-x) \big) - \nabla V(x) , y-x \rangle| \,}{|x-y| \leq \eta}{\gamma(x,y)}}{s}{0}{1}  \\
& \hspace{0.4cm} + \INTSeg{\INTDom{|\langle \nabla V \big( x + s(y-x) \big) - \nabla V(x) , y-x \rangle| \,}{|x-y| >\eta}{\gamma(x,y)}}{s}{0}{1} \\
& \leq \tfrac{\epsilon}{2} \hspace{-0.1cm} \INTDom{|x-y|}{\R^{2d}}{\gamma(x,y)} + 2 \hspace{-0.2cm} \sup_{z \in \co B_{\mu}(R)} |\nabla V(z)| \sup_{x,y \in B_{\mu}(R)}|x-y| \gamma \Big( \Big\{ (x,y) \in \R^{2d} ~\text{s.t.}~ |x-y| > \eta \Big\} \Big) \\
& \leq \tfrac{\epsilon}{2} \hspace{-0.1cm} \INTDom{|x-y|}{\R^{2d}}{\gamma(x,y)} + \tfrac{C_R}{\eta^2} \hspace{-0.1cm} \INTDom{|x-y|^2}{\R^{2d}}{\gamma(x,y)} \\
& \leq \tfrac{\epsilon}{2} W_2(\mu,\nu) + \tfrac{C_R}{\eta^2} W_2(\mu,\nu)^2,
\end{aligned}
\end{equation*}
for some constant $C_R > 0$ depending only on $\supp(\mu)$, $R >0$ and $V(\cdot)$, and where we used H\"older's and Chebyshev's inequalities together with the fact that $\gamma \in \Gamma_o(\mu,\nu)$. Thus whenever $W_2(\mu,\nu) \leq \eta' := \tfrac{\eta^2 \epsilon}{2 C_R}$, it holds that $|I_R| \leq \epsilon W_2(\mu,\nu)$ which precisely amounts to saying that $I_R = o_R(W_2(\mu,\nu))$. Since $\nabla V(\cdot) \in \Tan_{\mu} \Pcal_2(\R^d)$,  \eqref{eq:Appendix_TaylorV} implies by Proposition \ref{prop:Chainrule} that $\V(\cdot)$ is locally differentiable at $\mu$.

Let us now consider the functional $\Lpazo(\cdot)$ defined in \eqref{eq:Ldef}. By adapting the same arguments as above, one has for every $\gamma \in \Gamma_o(\mu,\nu)$ 
\begin{equation}
\label{eq:Appendix1}
\begin{aligned}
\Lpazo(\nu) - \Lpazo(\mu) & = \INTDom{\Big( L(y,\nu)-L(x,\mu) \Big)}{\R^{2d}}{\gamma(x,y)} \\
& = \INTDom{\Big( \langle \nabla_x L(x,\mu),y-x \rangle + \big( L(x,\nu) - L(x,\mu) \big) \Big)}{\R^{2d}}{\gamma(x,y)} + o_R(W_2(\mu,\nu)),
\end{aligned}
\end{equation}
where we used the fact that $\langle \nabla_x L(x,\nu),x-y \rangle = \langle \nabla_x L(x,\mu),x-y \rangle + o(|x-y|)$ as $W_2(\mu,\nu) \rightarrow 0$. Since $\mu \in \Pcal_c(\R^d) \mapsto L(x,\mu)$ is locally differentiable for $\mu$-almost every $x \in \R^d$, it further holds
\begin{equation}
\label{eq:Appendix2}
L(x,\nu) - L(x,\mu) = \INTDom{\langle \nabla_{\mu} L(x,\mu)(z) , w-z \rangle}{\R^{2d}}{\gamma(z,w)} + o_R(W_2(\mu,\nu)),
\end{equation}
so that by plugging \eqref{eq:Appendix2} into \eqref{eq:Appendix1} and applying Fubini's Theorem, we can derive
\begin{equation}
\label{eq:AppendixTalorL}
\Lpazo(\nu) - \Lpazo(\mu) = \INTDom{\Big\langle \nabla_x L(x,\mu) + \INTDom{\nabla_{\mu} L(z,\mu)(x)}{\R^d}{\mu(z)} , y-x \Big\rangle}{\R^{2d}}{\gamma(x,y)} + o_R(W_2(\mu,\nu)).
\end{equation}
Finally observe that $\nabla_x L(\cdot,\mu) \in \Tan_{\mu} \Pcal_2(\R^d)$ by definition, and also
\begin{equation*}
\INTDom{\nabla_{\mu} L(z,\mu)(\cdot)}{\R^d}{\mu(z)} \in \Tan_{\mu} \Pcal_2(\R^d),
\end{equation*}
because $\nabla_{\mu} L(z,\mu)(\cdot) \in \Tan_{\mu} \Pcal_2(\R^d)$ for any $z \in \supp(\mu)$ and
\begin{equation*}
\INTDom{\Big\langle \INTDom{\nabla_{\mu} L(z,\mu)(x)}{\R^d}{\mu(z)} , \xi(x) \Big\rangle}{\R^d}{\mu(x)} = \INTDom{\INTDom{\langle \nabla_{\mu} L(z,\mu)(x),\xi(x)\rangle}{\R^d}{\mu(x)}}{\R^d}{\mu(z)} = 0,
\end{equation*}
whenever $\xi \in \Tan_{\mu} \Pcal_2(\R^d)^{\perp}$, by linearity of the integral and Fubini's Theorem. Therefore, the identity \eqref{eq:AppendixTalorL} implies together with Proposition \ref{prop:Chainrule} that $\Lpazo(\cdot)$ is locally differentiable at $\mu$.
\end{proof}

\begin{rmk}[Global definition of Wasserstein gradient]
In general, the gradient $\nabla \phi(\mu)$ of a functional $\phi : \Pcal_c(\R^d) \rightarrow \R$ at $\mu \in \Pcal_c(\R^d)$ is an element of $L^2(\R^d,\R^d;\mu)$, and is thus well-defined on $\supp(\mu)$. In the particular case where $\nabla \phi(\mu)$ is a continuous map -- as in Proposition \ref{prop:Example} above -- , it is possible to extend it as a continuous function defined over the whole of $\R^d$ and not only $\supp(\mu)$. 
\end{rmk}

\smallskip

\begin{flushleft}
{\small{\bf  Acknowledgement.}  This material is based upon work supported by the Air Force Office of Scientific Research under award number FA9550-18-1-0254.}
\end{flushleft}


\bibliographystyle{plain}
{\footnotesize
\bibliography{../ControlWassersteinBib}
}

\end{document}